\documentclass[a4paper]{elsarticle}

\usepackage[utf8]{inputenc}
\usepackage[english]{babel}
\usepackage{enumerate}
\usepackage{latexsym}
\usepackage{amsthm,amsmath}
\usepackage{amssymb}
\usepackage{aliascnt} 
\usepackage{multicol}
\usepackage{tikz}
\usepackage{float} 

\usepackage[bookmarks]{hyperref} 
\hypersetup{
	pdftitle={The 2-switch-degree of a graph},
	pdfauthor={V. N. Schvöllner, A. Pastine},
	colorlinks=true,
	linkcolor=blue,
	citecolor=red,
	filecolor=cyan,
	urlcolor=magenta
}

\usetikzlibrary{babel,decorations.pathreplacing,calc,positioning,shapes.geometric}

\usepackage{cleveref}

\definecolor{10}{RGB}{115,59,171}
\definecolor{8}{RGB}{212,122,240}
\definecolor{7}{RGB}{99,212,119}
\definecolor{6}{RGB}{183,240,164}
\definecolor{D}{RGB}{255,162,79}
\definecolor{E}{RGB}{255,84,0}
\definecolor{F}{RGB}{158,248,255}
\definecolor{G}{RGB}{128,135,255}
\definecolor{I}{RGB}{187,255,0}
\definecolor{A}{cmyk}{.9,.05,.4,0}
\definecolor{B}{RGB}{150,30,150}
\definecolor{C}{RGB}{186,155,189}
\definecolor{9}{RGB}{0,180,60}
\definecolor{0}{RGB}{30,123,191}
\definecolor{1}{RGB}{255,113,102}
\definecolor{2}{RGB}{41,199,92}
\definecolor{3}{RGB}{242,207,16}
\definecolor{5}{RGB}{255,15,154}
\definecolor{4}{rgb}{.8,0,.8}

\definecolor{Red}{rgb}{1,0.4,0.4}
\definecolor{Green}{rgb}{.1,.5,.1}
\definecolor{Blue}{rgb}{.1,.1,.5}
\definecolor{blue}{RGB}{0,0,255}
\definecolor{Yellow}{rgb}{.8,.4,0}
\definecolor{X}{rgb}{.8,.4,0}
\definecolor{H}{rgb}{0,0,1}
\definecolor{light}{rgb}{.67,.84,.90}
\definecolor{Cyan}{rgb}{0,1,1}
\definecolor{Purple}{rgb}{.5,0,.5}
\definecolor{Purple2}{rgb}{.5,.2,.5}
\definecolor{white}{rgb}{1.0,1.0,1.0}
\definecolor{Purple2}{rgb}{.8,.4,0}
\definecolor{Amarillo}{RGB}{225,191,73}
\definecolor{Celeste}{RGB}{117,170,219}
\definecolor{Castano}{RGB}{232,53,17}
\definecolor{Black}{RGB}{0,0,0}
\definecolor{White}{RGB}{255,255,255}
\definecolor{gris}{rgb}{.5,.5,.5}

\newtheorem{theorem}{Theorem}[section]
\newaliascnt{corollary}{theorem}
\newtheorem{corollary}[corollary]{Corollary}
\aliascntresetthe{corollary}
\newaliascnt{lemma}{theorem}
\newtheorem{lemma}[lemma]{Lemma}
\aliascntresetthe{lemma}
\newaliascnt{proposition}{theorem}
\newtheorem{proposition}[proposition]{Proposition}
\aliascntresetthe{proposition}

\DeclareMathOperator{\ext}{ext}
\DeclareMathOperator{\dist}{dist}
\DeclareMathOperator{\diam}{diam}
\DeclareMathOperator{\act}{act}
\DeclareMathOperator{\dpe}{dpe}
\DeclareMathOperator{\ecc}{ecc}

\pgfdeclarelayer{background2}
\pgfdeclarelayer{background}
\pgfdeclarelayer{foreground}
\pgfsetlayers{background2,background,main,foreground}

\begin{document}
		\begin{frontmatter}
\title {
The 2-switch-degree of a graph
}
\address[IMASL]{Instituto de Matem\'atica Aplicada San Luis, Universidad Nacional de San Luis and CONICET, San Luis, Argentina.}
\address[DEPTO]{Departamento de Matem\'atica, Universidad Nacional de San Luis, San Luis, Argentina.}

\author[IMASL]{Victor N. Schvöllner}\ead{vnsi9m6@gmail.com}
\author[IMASL,DEPTO]{Adri\'an Pastine}\ead{agpastine@unsl.edu.ar}

		\begin{abstract}
			In this work, we study the 2-switch-degree of a graph $G$, that is, the degree of $G$ as a vertex of the realization graph $\mathcal{G}(d)$ associated with the degree sequence $d$ of $G$. We characterize the active and inactive vertices of a graph, with special attention to the case of split graphs, which play a central role in this setting by Tyshkevich decomposition. We establish the basic properties of the degree, showing in particular that it is additive with respect to the Tyshkevich composition. We then give an explicit formula for the degree in terms of $2K_2$-subgraphs, $C_4$-subgraphs and triangles, which yields an $O(n^3)$ algorithm for its computation and reveals an unexpected connection with the first and second Zagreb indices from Chemical Graph Theory. Finally, we obtain explicit formulas for the degree of trees and unicyclic graphs, and we show that the subgraph of $\mathcal{G}(d)$, induced by the trees with degree sequence $d$, is regular.
		\end{abstract}
\end{frontmatter}	

\section{Introduction}\label{sec:introduction}

We begin with a brief intuitive description of the key concepts studied in this paper, so that we can outline our main results. For formal definitions see \Cref{sec:prelim}. A $2$-switch on a graph $G$ is the exchange of two disjoint edges $ab$, $cd$ with non-edges $ac$, $bd$. 
We say that a vertex $a$ is \emph{active} in a graph $G$ if it participates in some 2-switch on $G$. The degree sequence of a graph $G$ with $V(G) = [n]=\{1,2,\ldots,n\}$ is $d=d(G)=(d_1,\ldots,d_n)$ where $d_i$ is the degree of vertex $i$ in $G$.

In \Cref{sec:vertices.activos} we study the set of active vertices of a graph (i.e., the set of vertices that participate in a $2$-switch on the graph). In \Cref{2switch.preservs.act} we show that the set of active vertices of a graph depends only on its degree sequence. That is, if $G$ and $H$ are graphs with vertex set $[n]$ with the same degree sequence, then they have the same active set. Using this result, we establish in \Cref{regular.graphs} that every connected regular non-complete graph is active (i.e., all its vertices are active). Next, we explore the relationship between the active set of a graph and its diameter. In this direction, we show that if $G$ is a non-active graph without isolated vertices, then all the graphs with the same degree sequence as $G$ have to be connected (\Cref{inactive.implies.connected}). As a consequence, if $G$ has no isolated vertices and diameter $\geq 4$, then $G$ is active (\Cref{inactive.implies.diam<=3.contrarr}). 

In \Cref{sec:act_vert_split}, we first discuss various properties of active vertices in split graphs, highlighting the characterization of their inactive vertices (\Cref{split.caract.vert.activos}). Next, for a large class of balanced and active split graphs, we provide a simple criterion to determine if they are indecomposable under Tyshkevich composition (\Cref{indecomposable_test_split} and \Cref{prime_test_split}).

Let $d$ be a graphical sequence. The realization graph of $d$ is the graph ${\cal G}(d)$ whose vertices are the graphs with degree sequence $d$, where two graphs $H_1,H_2$ are adjacent if $H_1$ is obtained from $H_2$ by a $2$-switch. In \Cref{sec:espacio.activo} we show that if $G$ is a graph with degree sequence $d$, $G^*$ is the subgraph of $G$ induced by its active vertices, and $d^*$ is the degree sequence of $G^*$, then ${\cal G}(d)$ is isomorphic to ${\cal G}(d^*)$ (\Cref{G(d)_iso_G(d*)}).

The degree of a graph $G$ is its degree as a vertex of ${\cal G}(d(G))$. This notion captures the flexibility of a graph under local transformations and is closely tied to other combinatorial and structural properties. In \Cref{sec:prop.basicas.deg}, after having established some basic facts about the degree (such as its relationship with induced subgraphs, and a formula for the degree of a graph in terms of the degrees and sizes of its connected components), we show in \Cref{deg(SoG)=deg(S)+deg(G)} that this parameter is additive with respect to Tyshkevich composition (the degree of the Tyshkevich composition is the sum of the degrees). Finally, using this result, we prove that a Tyshkevich composition is active if and only if each factor is (\Cref{SoG.act.iff.S_G.act}).

In \Cref{sec:formulas.explicitas} we give an explicit formula (\Cref{deg.general.formula}) for the graph degree in terms of the numbers of pairs of disjoint edges, $C_4$-subgraphs and triangles, from which it follows that the degree can be computed in $O(n^3)$ time, where $n$ is the order of the graph (indeed, within the cost of a single $n\times n$ matrix multiplication). Along the way, we establish connections with the first and second Zagreb indices from Chemical Graph Theory (\Cref{relacion.deg.zeta2}).

Finally, in \Cref{sec:deg.arboles.unicicl} we focus on the subgraph ${\cal T}(d)$ of ${\cal G}(d)$, induced by all trees with degree sequence $d$. We give a nice explicit formula for the degree of a tree in ${\cal T}(d)$ and ${\cal G}(d)$ (\Cref{deg_f.tree} and \Cref{deg.tree.G(d)}). Remarkably, ${\cal T}(d)$ turns out to be a regular graph (\Cref{T_d_regular}). Similarly, we study the subgraph ${\cal U}(d)$ of ${\cal G}(d)$, induced by all unicyclic graphs with degree sequence $d$. Again, we give explicit formulas for the degree of a unicyclic graph in ${\cal U}(d)$ and ${\cal G}(d)$ (\Cref{deg_u.unicyclic} and \Cref{deg_u.explicito}).


\section{Preliminaries}\label{sec:prelim}	

In this section we present our notation and some historical context. For specific graph theoretical concepts not defined here, we direct the reader to \cite{chartrand2010graphs}.
Let $G$ be a graph. We use the notation $V(G)$ and $E(G)$ to refer to the vertex set and the edge set of $G$, respectively. Thus, $|V(G)|$ and $|E(G)|$ denote, respectively, the order and the size of $G$. The complement of $G$ is denoted by $\overline{G}$. The expression $H\subseteq G$ means that $H$ is a subgraph of $G$. On the other hand $H\preceq G$ means that $H$ is an induced subgraph of $G$. As usual, if $W\subseteq V(G)$, we denote by $G[W]$ the subgraph of $G$ induced by $W$. Let $v\in V(G)$. The set $N_{G}(v)=\{x\in V(G): vx\in E(G)\}$ is the (open) \emph{neighborhood} of $v$ in $G$. The \emph{degree} of a vertex $v$ in $G$ (i.e., $|N_G(v)|$) is denoted by $\deg_G(v)$. If $G$ is clear from the context, we use $N_v$ instead of $N_G(v)$ and $d_v$ instead of $\deg_G(v)$. As mentioned in \Cref{sec:introduction}, we usually write $V(G) = [n]$ for some $n$. Then, the degrees of the vertices in $G$ can be arranged as an $n$-tuple $(d_v)_{v=1}^n = (d_1, \dots, d_n)$, forming what is called the \emph{degree sequence} of $G$, denoted by $d(G)$ or simply by $d$, if $G$ is clear from the context. When $d_1\geq d_2\geq\ldots\geq d_n$, we write $d$ in the more compact form $d_1^{\alpha_1}d_2^{\alpha_2}\ldots d_n^{\alpha_n}$, where $\alpha_i$ is the multiplicity of $d_i$ as an element of the multiset $\{d_i:i\in [n]\}$.
The notation $G\approx H$ means that the graphs $G$ and $H$ are isomorphic.\\


Let $G$ be a graph containing four distinct vertices $a,b,c,d$ such that $ab,cd \in G$ and $ac,bd \notin G$. A \emph{2-switch} on $G$ is a process that removes the edges $ab$ and $cd$ from $G$ and adds the edges $ac$ and $bd$ (see \cite{chartrand2010graphs}). We denote it by the $2\times 2$ matrix ${{a \ b}\choose{c \ d}}$, where the rows represent the edges to be removed and the columns represent the edges to be added. If $\tau= {{a \ b}\choose{c \ d}}$ is a 2-switch on $G$, then $\tau(G)$ denotes the transformed graph $(G-\{ab,cd\})+ \{ac,bd\}$. If $\theta=(\tau_i)_{i=1}^k$ is a sequence of 2-switches applied to $G$, we denote the transformed graph by $\theta(G)$ or $\tau_k\ldots\tau_1(G)$. By the symbol $\tau^{-1}$ we refer to the 2-switch ${a \ c}\choose{b \ d}$ on $\tau(G)$, which ``undo'' $\tau$. So, $\tau^{-1}$ satisfies that $\tau^{-1}\tau(G)=G$. Note that ${a \ c}\choose{b \ d}$, as a matrix, is the transpose of  ${{a \ b}\choose{c \ d}}$. A key observation is that, if ${{a \ b}\choose{c \ d}}$ is a 2-switch on $G$, then $G[a,b,c,d]$ is isomorphic to $P_4, C_4$ or $2K_2$. This can be quickly checked by analyzing the 11 possible isomorphism classes for a graph of order 4 (see \Cref{los.11.de.orden4}).
\begin{figure}[H]
	\[
	\begin{tikzpicture}
		[scale=.7,auto=left,every node/.style={scale=.65,circle,thick,draw}] 
		\node (n1) at (0,0) {};
		\node (n2) at (1,0)  {};
		\node [label=$\overline{K_4}$](n3) at (0,1)  {};
		\node (n4) at (1,1) {};
		\foreach \from/\to in {}
		\draw (\from) -- (\to);
		
		\node (n5) at (3,0) {};
		\node (n6) at (4,0)  {};
		\node [label=$\overline{D_4}$](n7) at (3,1)  {};
		\node (n8) at (4,1) {};
		\foreach \from/\to in {n5/n6}
		\draw (\from) -- (\to);
		
		\node (n9) at (6,0) {};
		\node (n10) at (7,0)  {};
		\node [label=$2K_2$](n11) at (6,1)  {};
		\node (n12) at (7,1) {};
		\foreach \from/\to in {n9/n10,n11/n12}
		\draw (\from) -- (\to);
		
		\node (n13) at (9,0) {};
		\node (n14) at (10,0)  {};
		\node [label=$\overline{U_4}$](n15) at (9,1)  {};
		\node (n16) at (10,1) {};
		\foreach \from/\to in {n13/n14,n14/n16}
		\draw (\from) -- (\to);
		
		\node (n17) at (12,0) {};
		\node (n18) at (13,0)  {};
		\node [label=$P_4$](n19) at (12,1)  {};
		\node (n20) at (13,1) {};
		\foreach \from/\to in {n17/n18,n20/n19,n18/n20}
		\draw (\from) -- (\to);
		
		\node (n21) at (15,0) {};
		\node (n22) at (16,0)  {};
		\node [label=$S_4$](n23) at (15,1)  {};
		\node (n24) at (16,1) {};
		\foreach \from/\to in {n21/n22,n22/n24,n22/n23}
		\draw (\from) -- (\to);
		
	\end{tikzpicture}
	\]
	\[
	\begin{tikzpicture}
		[scale=.7,auto=left,every node/.style={scale=.65,circle,thick,draw}] 
		\node (n1) at (0,0) {};
		\node (n2) at (1,0)  {};
		\node [label=$\overline{S_4}$](n3) at (0,1)  {};
		\node (n4) at (1,1) {};
		\foreach \from/\to in {n1/n2,n2/n4,n4/n1}
		\draw (\from) -- (\to);
		
		\node (n5) at (3,0) {};
		\node (n6) at (4,0)  {};
		\node [label=$U_4$](n7) at (3,1)  {};
		\node (n8) at (4,1) {};
		\foreach \from/\to in {n5/n6,n6/n8,n8/n5,n5/n7}
		\draw (\from) -- (\to);
		
		\node (n9) at (6,0) {};
		\node (n10) at (7,0)  {};
		\node [label=$C_4$](n11) at (6,1)  {};
		\node (n12) at (7,1) {};
		\foreach \from/\to in {n9/n10,n11/n12,n10/n12,n9/n11}
		\draw (\from) -- (\to);
		
		\node (n13) at (9,0) {};
		\node (n14) at (10,0)  {};
		\node [label=$D_4$](n15) at (9,1)  {};
		\node (n16) at (10,1) {};
		\foreach \from/\to in {n13/n14,n14/n16,n16/n15,n15/n13,n14/n15}
		\draw (\from) -- (\to);
		
		\node (n17) at (12,0) {};
		\node (n18) at (13,0)  {};
		\node [label=$K_4$](n19) at (12,1)  {};
		\node (n20) at (13,1) {};
		\foreach \from/\to in {n17/n18,n20/n19,n18/n20,n20/n17,n17/n19,n18/n19}
		\draw (\from) -- (\to);
	\end{tikzpicture} 
	\]
	\caption{The 11 unlabeled graphs of order 4.}
	\label{los.11.de.orden4}
\end{figure}
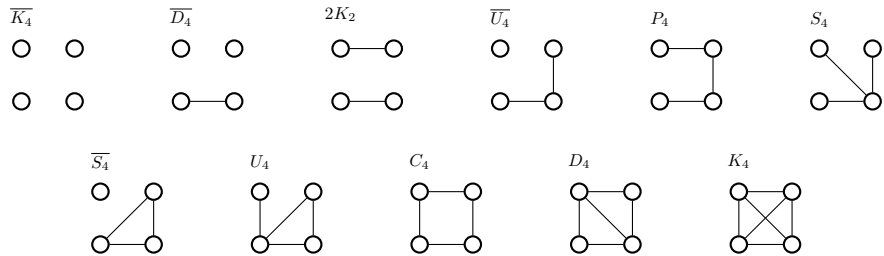
An elementary property of 2-switches (which directly follows from the definition) is that they preserve the degree sequence, i.e., $d(\tau(G))=d(G)$ for every 2-switch $\tau$ on $G$. From it, the following fundamental theorem can be proven.

\begin{theorem}[\cite{chartrand2010graphs}]
	\label{berge's.theorem}
	If $G$ and $H$ are two distinct graphs with the same degree sequence, then there exists a sequence of 2-switches $(\tau_i)_{i=1}^k$ such that $H=\tau_k\ldots\tau_1(G)$. 
\end{theorem}

\Cref{berge's.theorem} is one of the earliest results in this area. Although there is no clear original bibliography, it is probably due to Berge. \\

The \emph{realization graph} (see \cite{arikati1999realization}) $\mathcal{G}(d)$, associated with a degree sequence $d$, is the graph such that 
\begin{enumerate}[(1).]
	\item $V(\mathcal{G}(d))$ is the set of all labeled graphs with degree sequence $d$; 
	
	\item $\{G,H\}\in E(\mathcal{G}(d))$ if and only if $H=\tau(G)$ for some 2-switch $\tau$.
\end{enumerate}
In \Cref{ejemplo_esp_trans} we can see a concrete example of this. 
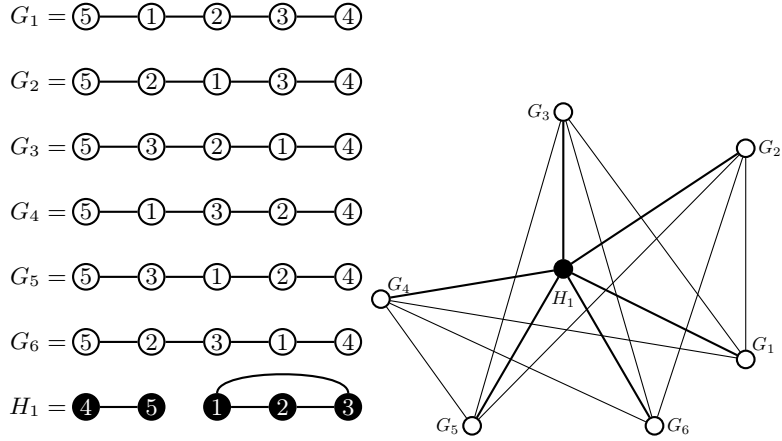
\begin{figure}[h] 
	\centering
	\begin{tikzpicture}[scale=1.15, every node/.style={circle, draw, thick, minimum size=1pt, inner sep=.8pt,font=\small}]
		
		\def\xsep{0.75}  
		\def\ysep{0.75}  
		
		\node (n11) at (1*\xsep, 0) [label=left:{$G_1=$}]{5};
		\node (n12) at (2*\xsep, 0) {1};
		\node (n13) at (3*\xsep, 0) {2};
		\node (n14) at (4*\xsep, 0) {3};
		\node (n15) at (5*\xsep, 0) {4};
		\draw[thick] (n11) -- (n12) -- (n13) -- (n14) -- (n15);
		
		\node (n21) at (1*\xsep, -1*\ysep)[label=left:{$G_2=$}] {5};
		\node (n22) at (2*\xsep, -1*\ysep) {2};
		\node (n23) at (3*\xsep, -1*\ysep) {1};
		\node (n24) at (4*\xsep, -1*\ysep) {3};
		\node (n25) at (5*\xsep, -1*\ysep) {4};
		\draw[thick] (n21) -- (n22) -- (n23) -- (n24) -- (n25);
		
		\node (n31) at (1*\xsep, -2*\ysep)[label=left:{$G_3=$}] {5};
		\node (n32) at (2*\xsep, -2*\ysep) {3};
		\node (n33) at (3*\xsep, -2*\ysep) {2};
		\node (n34) at (4*\xsep, -2*\ysep) {1};
		\node (n35) at (5*\xsep, -2*\ysep) {4};
		\draw[thick] (n31) -- (n32) -- (n33) -- (n34) -- (n35);
		
		\node (n41) at (1*\xsep, -3*\ysep)[label=left:{$G_4=$}] {5};
		\node (n42) at (2*\xsep, -3*\ysep) {1};
		\node (n43) at (3*\xsep, -3*\ysep) {3};
		\node (n44) at (4*\xsep, -3*\ysep) {2};
		\node (n45) at (5*\xsep, -3*\ysep) {4};
		\draw[thick] (n41) -- (n42) -- (n43) -- (n44) -- (n45);
		
		\node (n51) at (1*\xsep, -4*\ysep)[label=left:{$G_5=$}] {5};
		\node (n52) at (2*\xsep, -4*\ysep) {3};
		\node (n53) at (3*\xsep, -4*\ysep) {1};
		\node (n54) at (4*\xsep, -4*\ysep) {2};
		\node (n55) at (5*\xsep, -4*\ysep) {4};
		\draw[thick] (n51) -- (n52) -- (n53) -- (n54) -- (n55);
		
		\node (n61) at (1*\xsep, -5*\ysep)[label=left:{$G_6=$}] {5};
		\node (n62) at (2*\xsep, -5*\ysep) {2};
		\node (n63) at (3*\xsep, -5*\ysep) {3};
		\node (n64) at (4*\xsep, -5*\ysep) {1};
		\node (n65) at (5*\xsep, -5*\ysep) {4};
		\draw[thick] (n61) -- (n62) -- (n63) -- (n64) -- (n65);
		
		\node [color=black,fill=black](n71) at (1*\xsep, -6*\ysep)[label=left:{\textcolor{black}{$H_1=$}}] {\textcolor{white}{4}};
		\node [color=black,fill=black](n72) at (2*\xsep, -6*\ysep) {\textcolor{white}{5}};
		\node [color=black,fill=black](n73) at (3*\xsep, -6*\ysep) {\textcolor{white}{1}};
		\node [color=black,fill=black](n74) at (4*\xsep, -6*\ysep) {\textcolor{white}{2}};
		\node [color=black,fill=black](n75) at (5*\xsep, -6*\ysep) {\textcolor{white}{3}};
		\draw[thick,color=black] (n71) -- (n72)  (n73) -- (n74) -- (n75);
		\draw[thick,color=black] (n73) .. controls +(0,0.45) and +(0,0.45) .. (n75);
		
	\end{tikzpicture} 
	\begin{tikzpicture}[scale=.8, every node/.style={scale=0.7}]
		\node[draw, circle, fill=white, thick] (G_1) at (3.00,-1.50) {};
		\node[above right] at (G_1) {$G_1$};
		\node[draw, circle, fill=white, thick] (G_2) at (3.00,2) {};
		\node at ($(G_2)+(0.4,0)$) {$G_2$};
		\node[draw, circle, fill=white, thick] (G_3) at (0.00,2.60) {};
		\node at ($(G_3)+(-0.4,0)$) {$G_3$};
		\node[draw, circle, fill=white, thick] (G_4) at (-3.00,-0.50) {};
		\node[above right] at (G_4) {$G_4$};
		\node[draw, circle, fill=white, thick] (G_5) at (-1.50,-2.60) {};
		\node at ($(G_5)+(-0.4,0)$) {$G_5$};
		\node[draw, circle, fill=white, thick] (G_6) at (1.50,-2.60) {};
		\node at ($(G_6)+(0.4,0)$) {$G_6$};
		\node[draw, circle, fill=black, thick,color=black] (H_1) at (0.00,0.00) {};
		\node at ($(H_1)+(0,-0.5)$) {\textcolor{black}{$H_1$}};
		\draw (G_1) -- (G_2);
		\draw (G_1) -- (G_3);
		\draw[thick,color=black] (G_1) -- (H_1);
		\draw (G_1) -- (G_4);
		\draw[thick,color=black] (G_2) -- (H_1);
		\draw (G_2) -- (G_5);
		\draw (G_2) -- (G_6);
		\draw[thick,color=black] (G_3) -- (H_1);
		\draw (G_3) -- (G_5);
		\draw (G_3) -- (G_6);
		\draw[thick,color=black] (H_1) -- (G_4);
		\draw[thick,color=black] (H_1) -- (G_5);
		\draw[thick,color=black] (H_1) -- (G_6);
		\draw (G_4) -- (G_5);
		\draw (G_4) -- (G_6);
	\end{tikzpicture}
	\caption{The realization graph $\mathcal{G}(2^31^2)$.}
	\label{ejemplo_esp_trans}
\end{figure}
A direct consequence of \Cref{berge's.theorem} is that $\mathcal{G}(d)$ is connected. In recent years, the problem of determining whether a given induced subgraph of $\mathcal{G}(d)$ is connected or not has been extensively studied. In this direction, many advances have already been made. For example, it is known that $\mathcal{F}(d)$, i.e., the subgraph of $\mathcal{G}(d)$ induced by forests, is connected (see \cite{schvollner_pseudoforests}). An analogous result also holds for unicyclic graphs, connected graphs, and pseudoforests, i.e., graphs whose components are trees or unicyclic graphs (see \cite{schvollner_pseudoforests,taylor2006contrained}). The most important feature of a connected $\mathcal{H}\preceq\mathcal{G}(d)$ is that it ensures the possibility of transforming, via 2-switches, a graph $G\in V(\mathcal{H})$ into another graph $G'\in V(\mathcal{H})$ in such a way that every intermediate graph in the transition also belongs to $V(\mathcal{H})$. If $d=(d_{v})_{v=1}^{n}$ is the degree sequence of $G$, then the sequence $\overline{d}=(n-1-d_{v})_{v=1}^{n}$ is the degree sequence of $\overline{G}$. The following result is stated, without a detailed proof, in \cite{barrus2023cliques} (page 2). Since we have not found an explicit proof in the literature, we include our own for completeness.

\begin{theorem}[\cite{barrus2023cliques}]
	\label{dual.spaces.iso} 
	If $d$ is a degree sequence, then $\mathcal{G}(\overline{d}) \approx \mathcal{G}(d)$.
\end{theorem}

	\begin{proof}
		Let $\varphi: V(\mathcal{G}(d)) \rightarrow V(\mathcal{G}(\overline{d}))$ be defined by $\varphi(X)=\overline{X}$. Since complementation is an involution, $\varphi$ is a bijection with $\varphi^{-1}=\varphi$. Moreover, if $H=\tau(G)$ for some 2-switch $\tau$ on $G$, then $\overline{H}=\tau^{-1}(\overline{G})$, so $\varphi$ maps edges of $\mathcal{G}(d)$ to edges of $\mathcal{G}(\overline{d})$. As $\varphi^{-1}=\varphi$ is of the same form, the same argument shows it also preserves edges. Hence, $\varphi$ is an isomorphism.
	\end{proof}

Let $G$ be a graph with degree sequence $d$. The \emph{2-switch-degree} of $G$ (or simply, the degree), denoted by $\deg(G)$, is the degree of $G$ viewed as a vertex of $\mathcal{G}(d)$, i.e., the number of neighbors of the vertex $G$ in the realization graph associated with $d$. In other words, $\deg(G)$ is the number of 2-switches that we can perform on $G$. As mentioned in \Cref{sec:introduction}, in this article we focus on the study of this parameter. 

A key family of graphs for the study of realization graphs is that of split graphs. A \emph{split graph}
is a graph $G$ whose set of vertices can be partitioned into two sets $K$ and $I$, such that $K$ is a clique  and $I$ is an independent set. Split graphs have been characterized as those that do not contain a $C_4$, $C_5$ or $2K_2$ as induced subgraphs (see \cite{hammer1977split}). If $S\in V(\mathcal{G}(d))$ and $S$ is split, then all the vertices of $\mathcal{G}(d)$ are also split (see \cite{tyshkevich2000decomposition}).

The concept of realization graph has received increasing attention over the years. The study of these objects was formalized in \cite{eggleton2006realizations}, and subsequent work (see \cite{arikati1999realization}) showed that realization graphs of split graphs are isomorphic to interchange graphs of certain matrices. Tyshkevich introduced in \cite{tyshkevich2000decomposition} a decomposition of graphs into indecomposable factors which are all split graphs, except, at most, for one of them. Building on this, Barrus proved in \cite{barrus2016realization} that realization graphs of decomposable graphs can be expressed as the Cartesian product of the realization graphs of the factors, and more recently, cliques in realization graphs and degree sequences whose realization graphs are complete graphs were characterized in \cite{barrus2023cliques}.

	\section{Active vertices} \label{sec:vertices.activos}

A vertex $a$ is active if and only if there are three vertices $b,c,d$, such that $ab, cd$ are disjoint edges of $G$, and $ac, bd$ are non-edges of $G$. In such a case we say that the 2-switch ${{a \ b}\choose{c \ d}}$ \emph{activates} $a$. We denote by $\act(G)$ the set of all active vertices of $G$. We say that the vertex $v$ is \emph{inactive} in $G$ if $v\notin\act(G)$, i.e., if there is no 2-switch on $G$ activating $v$. Thus, $V(G)-\act(G)$ is the set of all inactive vertices of $G$. 
It is important to note that a vertex $ a $ is active in $ G $ if and only if $ a $ is active in $ \overline{G} $. Indeed, if $ \tau $ is a 2-switch that activates $ a $ in $ G $, then $ \tau^{-1} $ activates $ a $ in $ \overline{G} $. Thus,
\[ \act(G) = \act(\overline{G}). \]
From this, it immediately follows that universal vertices are inactive, since they are isolated in the complement graph. 
	
%
	
	The following notation will be useful from now on:
	\[ Q_G = \{ H \preceq G : |V(H)| = 4 \}, \quad Q_G^* = \{ H \in Q_G : \deg(H)\neq 0 \}. \]
	The letter ``$Q$'' stands for ``quartets'', as the members of $Q_G$ has order 4.
	
	A key fact about active and inactive vertices is that they are invariant under $2$-switch. I.e., a vertex $a$ is active in $G$ if and only if it is active in $\tau(G)$ for every 2-switch $\tau$ on $G$. 

	\begin{theorem}
		\label{2switch.preservs.act}
		Let $d$ be a graphical sequence. If $G,H\in V(\mathcal{G}(d))$, then $\act(G) = \act(H)$. In other words, the 2-switch preserves active and inactive vertices.
	\end{theorem}
	
	\begin{proof}
		A vertex is inactive in $G$ if and only if it lies in no induced $2K_2$, $P_4$, or $C_4$; equivalently, if and only if it is an isolated vertex of $A_4(G)$ (see \cite{barrus.west.A4}). Since a 2-switch preserves the degree sequence, $G$ and $\tau(G)$ have the same partition of $V(G)$ into components of the $A_4$-structure, by Corollary 3.14 of \cite{barrus.west.A4} (alternatively, this follows immediately from Proposition 3.12 and Theorem 3.2 of \cite{barrus.west.A4}). In particular, the trivial components of $A_4(G)$ and $A_4(\tau(G))$ coincide; that is, $A_4(G)$ and $A_4(\tau(G))$ have the same isolated vertices, hence the same inactive vertices. Therefore $\act(G)=\act(\tau(G))$.
	\end{proof}
	
	
	Given a graphical sequence $d$ and $G\in V(\mathcal{G}(d))$, \Cref{2switch.preservs.act} justifies the notation $\act(d)$ as an alternative to $\act(G)$.
	
	\begin{theorem}
		\label{same.degree.act.inact}
		If $a$ is an active (inactive) vertex of a graph $G$, then all vertices of $G$ with the same degree (in $G$) as $a$ are active (inactive) in $G$.
	\end{theorem}
	
	\begin{proof}
		Let $d=d(G)$, $x\in V(G)$ with $d_x = d_a$, and let $\varphi = (a \, x)$ be the transposition swapping $a$ and $x$. Let $H = \varphi(G)$ be the graph obtained from $G$ by relabeling its vertices through $\varphi$. Since $d_a = d_x$, we have $d(H)=d$, so $G,H\in V(\mathcal{G}(d))$ and \Cref{2switch.preservs.act} gives $\act(H)=\act(G)$. On the other hand, $\varphi$ is an isomorphism from $G$ to $H$, and isomorphisms map active vertices to active vertices. Hence, $a\in\act(G)$ if and only if $\varphi(a)=x\in\act(H)$. Combining both facts, $x\in\act(H)=\act(G)$ exactly when $a\in\act(G)$. As $x$ was an arbitrary vertex of degree $d_a$, every such vertex is active when $a$ is, and inactive when $a$ is.
	\end{proof}
	
	We say that a graph $ G $ is \emph{active} if $\act(G) = V(G)$. Otherwise, we simply say that $ G $ is \emph{not active} (or non-active). In particular, if $ \act(G) = \varnothing \neq V(G) $, we say that $ G $ is \emph{inactive}. The null graph $K_0=(\varnothing,\varnothing)$ is vacuously active, because $\act(K_0)=\varnothing=V(K_0)$. Apart from $K_0$, it is clear that $|V(G)|\geq 4$ when $G$ is active. Some examples of active graphs are $ C_n $ and $ P_n $, for $ n \geq 4 $. Also, all disconnected graphs without isolated vertices are clearly active. Inactive graphs, on the other hand, are precisely the non-null graphs of degree 0, i.e., threshold graphs (see \cite{mahadev1995threshold}). Thus, the property of being inactive is hereditary. Conversely, the property of being active is clearly non-hereditary, since $ K_1 \preceq G $ for any active graph $ G\neq K_0 $, but $ \act(K_1) \neq V(K_1) $. 
	
	The following theorem is interesting because it establishes that all connected regular graphs are active, except for complete graphs, which are inactive.
	
	\begin{theorem}
		\label{regular.graphs}
		Let $ G $ be a connected regular graph. If $ G $ is not complete, then $ G $ is active. 
	\end{theorem}
	
	\begin{proof}
		We prove the contrapositive. By \Cref{same.degree.act.inact}, all the vertices of $G$ are inactive, i.e., $G$ is a threshold graph. A connected threshold graph has a universal vertex (since the last digit of the binary sequence that determines it must be a 1, see \cite{mahadev1995threshold}). Therefore, $k=|V(G)|-1$, and so $G$ is complete.
	\end{proof}
	
	
	\begin{theorem}
		\label{inactive.implies.connected}
		Let $d$ be a degree sequence of a graph without isolated vertices. If $V(\mathcal{G}(d))$ contains a non-active graph, then all members of $V(\mathcal{G}(d))$ are connected. 
	\end{theorem}   
	
	\begin{proof}
		Suppose $V(\mathcal{G}(d))$ contains a disconnected graph $G$. Since each component of $G$ is non-trivial, $G$ is clearly active. Thus, by \Cref{2switch.preservs.act}, we conclude that every member of $V(\mathcal{G}(d))$ is active.
	\end{proof}
	
	If $G$ is a graph and $x,y\in V(G)$, we denote the distance between $x$ and $y$ in $G$ by $\dist_G(x,y)$, and the diameter of $G$ by $\diam(G)$. The \emph{eccentricity} $ \ecc_G(x) $ of a vertex $ x $ in $ G $ is a measure describing the greatest distance between $ x $ and any other vertex in $ G $. Formally, $\ecc_G(x) = \max\{ \dist_G(v, x) : v \in V(G) \}$.	In the following proposition, we show that in a graph without isolated vertices, inactive vertices have eccentricity 1 or 2.
	
	\begin{proposition}
		\label{eccen.inact.vertex}
		Let $ G $ be a graph without isolated vertices. If $ x \notin \act(G) $, then $ \dist_G(x, v) \leq 2 $, for every $ v \in V(G) $. In other words: $\ecc_G(x) \in \{1, 2\}$.
	\end{proposition}
	
	\begin{proof}
		If $ v $ is a vertex at distance $\ell$ from $ x $ (in $ G $), we can find an induced path of length $\ell$ connecting $ x $ to $ v $ in $ G $. If $ \ell\geq 3 $, then there exist vertices $ y $ and $ z $ such that $ \binom{x \, y}{z \, v} $ is a 2-switch on $ G $. This means $ x \in \act(G) $.
	\end{proof}
	
	\begin{corollary}
		\label{inactive.implies.diam<=3}
		Let $ G $ be a graph without isolated vertices. If $ G $ is not active, then $ \diam(G) \leq 3 $. Moreover, this bound is sharp. 
	\end{corollary}  
	
	\begin{proof}
		By \Cref{inactive.implies.connected}, we know $ G $ is connected. Consider a vertex $ x \notin \act(G) $. By the triangle inequality and \Cref{eccen.inact.vertex}, for any pair of vertices $ v, w \in G $, we have that  
		\begin{equation*}
			\dist_G(v, w) \leq \dist_G(v, x) + \dist_G(x, w) \leq 2 + 2 = 4.
		\end{equation*}
		By \Cref{eccen.inact.vertex}, equality holds if and only if $ \dist_G(v_0, x) = \dist_G(x, w_0) = 2 $ for some $ v_0, w_0 \in V(G) $. But in this case, $ G $ would contain an induced path of the form $v_0yxzw_0 \approx P_5$, for some $y,z\in V(G)$, forcing $ x $ to be active. Therefore, $ \dist_G(v, w) < 4 $ for every pair $ v, w \in V(G) $, and thus $ \diam(G) \leq 3 $. Finally, to see that the bound is sharp, consider the degree sequence $3^22^11^2$ (a $K_3$ with 2 leaves). 
	\end{proof}
	
	The converses of \Cref{inactive.implies.connected,inactive.implies.diam<=3} are not true: $ P_4 $ is a connected graph with diameter 3 where all its vertices are active. Even if the diameter of a graph $ G $ were 2, we cannot guarantee in general that $ G $ has any inactive vertices; for example, $ C_4 $ is an active graph with diameter 2. It is useful to rephrase \Cref{inactive.implies.diam<=3} in the following equivalent way.
	
	\begin{corollary}
		\label{inactive.implies.diam<=3.contrarr}
		Let $d$ be a graphical sequence without zeros. If $ V(\mathcal{G}(d)) $ contains a graph of diameter $ \geq 4 $, then all members of $ V(\mathcal{G}(d)) $ are active.
	\end{corollary} 
	
	
	\section{Active vertices in split graphs} \label{sec:act_vert_split}
	
	Recall from \Cref{sec:prelim} that, being $S$ a split graph, we can write $V(S)=K\dot{\cup}I$, where $K$ is a clique and $I$ is an independent set. In such a case, we say that the pair $(K,I)$ is a \emph{bipartition} for $S$. 
	The notation $(S,K,I)$ means that we are considering $S$ together with a bipartition $(K,I)$ for it. By $\omega(G)$ and $\alpha(G)$ we denote, respectively, the \emph{clique number} and the \emph{independence number} of a graph $G$. A split graph $(S,K,I)$ is said to be \emph{balanced} if $(|K|,|I|)=(\omega(S), \alpha(S))$, and \emph{unbalanced} otherwise. A split graph $S$ is balanced if and only if $S$ has a unique bipartition (see \cite{splitnordhausgaddum}). Two bipartitions $(K,I)$ and $(K',I')$ are equal if $K=K'$ and $I=I'$. A vertex $w$ of $S$ is said to be \emph{swing} if there exist bipartitions $(K,I), (K',I')$ for $S$ such that $w\in K$ and $w\in I'$. The set of swing vertices of $S$ is denoted by $W(S)$; $S$ is unbalanced if and only if $W(S)\neq\varnothing$ (see Section 4 of \cite{jaume2026nullspace}). 
	
	Observe that a vertex $ v $ is active in a split graph $ (S, K, I) $ if and only if there exists an $ H \preceq S $ such that $ v \in H $ and $ H \approx P_4 $. This is because split graphs do not contain induced subgraphs isomorphic to $ C_4 $ or $ 2K_2 $. If $ v \in I $, such an $ H $ will be of the form $vabc$, where $ a, b \in K $ and $ c \in I $. If instead $ v \in K $, $ H $ can be written as $avbc$, where $ a, c \in I $ and $ b \in K $.
	
	If $(S,K,I)$ is a split graph, define 
	\[ E(K)=\{xy\in E(S): x,y\in K\}. \]
	Note that $S-E(K)$ is a bipartite graph with the same bipartition as $S$.
	
	\begin{proposition}
		\label{split.caract.vert.activos}
		If $ (S, K, I) $ is a split graph, then the following holds:
		\begin{enumerate}
			\item Let $U$ be the set of universal vertices of $S$. If $I\neq\varnothing$, then $U=\bigcap_{v \in I} N_S(v)$ if and only if $S$ is not complete;
			
			\item $ u \in I $ is active in $ S $ if and only if there exists an $ x \in I $ such that $ N_S(x) - N_S(u) \neq \varnothing $ and $ N_S(u) - N_S(x) \neq \varnothing $;
			
			\item $ u \in I $ is inactive in $ S $ if and only if, for every $ v \in I $, either $ N_S(v) \subseteq N_S(u) $ or $ N_S(u) \subseteq N_S(v) $;
			
			\item $ u \in K $ is inactive in $ S $ if and only if, for every $ v \in K $, either $ N_S(v) \cap I \subseteq N_S(u) \cap I $ or $ N_S(u) \cap I \subseteq N_S(v) \cap I $; 
			
			\item A vertex of $ S $ is inactive in $ S $ if and only if its neighborhood in $S-E(K)$ is comparable by inclusion with all other neighborhoods in its partition;
			
			\item If $S$ is unbalanced and $ w\in W(S) $, then $ w \notin \act(S) $.
			
			\item If $ S $ is active, then $ S $ is balanced.
		\end{enumerate}
	\end{proposition}
	
	\begin{proof}
		\begin{enumerate}[(1).]
			\item ($\Leftarrow$) Suppose $S$ is not complete. We first claim that
				$U \subseteq K$. Indeed, if a universal vertex $u$ belonged to $I$, then,
				since $I$ is independent, we would have $I = \{u\}$; but then every two
				vertices of $K \cup u=V(S)$ would be adjacent, making $S$ complete, a
				contradiction. Now let $u \in U$. Since $u \notin I$, $u$ is adjacent to
				every $v \in I$, so $u \in \bigcap_{v \in I} N_v$. Reciprocally, assume
				$u \in \bigcap_{v \in I} N_v$. Then, $u \in K$, since $u$ is adjacent to
				some vertex of $I$ and no vertex of $I$ has neighbors in $I$. Hence,
				$N_u = (K - u) \,\dot\cup\, I$, which means that $u \in U$.
				
				($\Rightarrow$) We prove the contrapositive. Suppose $S$ is complete. Since
				$I$ is independent and $I \neq \varnothing$, we have $I = \{w\}$ for some $w$.
				Then $w \in U$, but $w \notin N_w = \bigcap_{v \in I} N_v$. Therefore,
				$U \neq \bigcap_{v \in I} N_v$.
			 
			\item If there exist $ a \in N_x - N_u $ and $ b \in N_u - N_x $, then $ubax\preceq S$, showing that $ u \in \act(S) $. Conversely, if $ u \in \act(S) $, then there exist $ a, b \in K $ and $ x \in I $ such that $ubax\preceq S$, so $ a \notin N_u $ and $ b \notin N_x $. Thus, both $ N_u - N_x $ and $ N_x - N_u $ are non-empty.
			
			\item This is simply the contrapositive of (2).
			
			\item If $u\in K$ is inactive in $S$, then, in $\overline{S}$, the same vertex is also inactive. Note that $(I,K)$ is a bipartition for $\overline{S}$. Thus, by (3), we have $ N_{\overline{S}}(v) \subseteq N_{\overline{S}}(u) $ or $ N_{\overline{S}}(u) \subseteq N_{\overline{S}}(v) $ for every $ v \in K $. Equivalently: $ N_v^c - v \subseteq N_u^c - u $ or $ N_u^c - u \subseteq N_v^c - v $, i.e., $ N_v \cup v \supseteq N_u \cup u $ or $ N_u \cup u \supseteq N_v \cup v $. Since $ N_x = (N_x \cap I) \cup (K - x) $ for every $ x \in K $, it follows that $ (N_v \cap I) \cup K \supseteq (N_u \cap I) \cup K $ or $ (N_u \cap I) \cup K \supseteq (N_v \cap I) \cup K $. Therefore, we conclude that $ N_v \cap I \supseteq N_u \cap I $ or $ N_u \cap I \supseteq N_v \cap I $.
			
			For the converse, suppose that $u\in \act(S)$. Then, there exist $ a, b \in I $ and $ x \in K $ such that $auxb\preceq S$, so $ a \notin N_x\cap I $ and $ b \notin N_u\cap I $. Thus, both $ (N_u\cap I) - (N_x\cap I) $ and $ (N_x\cap I) - (N_u\cap I) $ are non-empty.
			
			\item Combining (3) and (4) yields the result. 
			
			\item Consider the $ABC$-partition of $S$ (see Section 1 of \cite{splitnordhausgaddum}). By Theorem 4.4 of \cite{jaume2026nullspace}, we know that $W(S)=A$, so every swing vertex of $S$ is adjacent to all vertices of $B$ and non-adjacent to all vertices of $C$, by Theorem 2 of \cite{splitnordhausgaddum}. Moreover, $B$ is a clique of $S$, and $A$ is a clique or an independent set of $S$ of mutually twin vertices, again by Theorem 2 of \cite{splitnordhausgaddum}. Then, it is clear that any member of $Q_S$ containing a swing vertex cannot be isomorphic to $P_4$. Hence, $W(S)\cap\act(S)=\varnothing$.  
			
			\item If $ S $ is unbalanced, then $ S $ has a swing vertex, which is inactive by (6). Thus, $ \act(S) \neq V(S) $, which means that $S$ is not active.
		\end{enumerate}
	\end{proof} 
	
	It is important to note that the converse of statement (7) in \Cref{split.caract.vert.activos} is not true in general, as there exist balanced non-active graphs. Indeed, let $ S $ be the split graph with bipartition $ (K, I) = ([3], \{a, b, c\}) $, where $ N_a = \{1, 3\} $, $ N_b = \{2, 3\} $, and $ N_c = \{3\} $. It is easy to check by inspection that $(K,I)$ is the unique bipartition for $S$, showing that $ S $ is balanced. However, the vertices $ c $ and $ 3 $ are inactive in $ S $.
	
	\begin{proposition}
		\label{balanced_implies_|U|<|K|-1}
		Let $ (S, K, I) $ be a balanced split graph. Then, the following holds:
		\begin{enumerate}
			\item $ \deg(S) \geq 1 $, and $ |K|, |I| \geq 2 $; 
			\item $ K = \bigcup_{v \in I} N_S(v) $; 
			\item $ |U| \leq |K| - 2 $.
		\end{enumerate}
	\end{proposition}
	
	\begin{proof}
		\begin{enumerate}[(1).]
			\item If $ \deg(S) = 0 $, then $ S $ would be threshold and thus unbalanced (see \cite{splitnordhausgaddum} or Section 4 of \cite{jaume2026nullspace}). Hence, $ \deg(S) \geq 1 $, which implies that $ S $ contains an induced $ P_4 $. Consequently, $ |I|, |K| \geq 2 $. 
			\item Clearly, $ \bigcup_{v \in I} N_v \subseteq K $. If the inclusion were strict, there would be a vertex $ w \in K $ with no neighbors in $ I $. But then $(K-w,I\cup w)$ would be a bipartition for $S$ different from $(K,I)$, contradicting that $ S $ is balanced. 
			\item Since $ S $ is balanced, $\deg_S(v) < |K| $ for every $ v \in I $ (otherwise, $(K,I)$ would no longer be the unique bipartition for $S$). Thus, $ |U| < |K| $ by (2). To complete the proof, it suffices to show that $ |U| \neq |K| - 1 $. Suppose $ |U| = |K| - 1 $. Let $ x $ be the only vertex of $K$ satisfying that $\deg_S(x)<|V(S)|-1$. By (1), there exists $ w \in N_S(x) \cap I $. Clearly, since $S$ is balanced, $\overline{S}$ is balanced as well and so $(I,K)$ is its unique bipartition. Now, observe that $I-w\subseteq N_{\overline{S}}(w)$, $\overline{S}$ has $|K|-1$ isolated vertices, $\deg_{\overline{S}}(x)\neq 0$ and $ w x \notin E(\overline{S}) $. Thus, it must be $N_{\overline{S}}(w)=I-w$, which means that $ w $ has no neighbor in the independent part $K$ of $ \overline{S} $. Therefore, $(I-w, K\cup w)$ is a bipartition for $\overline{S}$ different from $(I,K)$, contradicting that $\overline{S}$ is balanced.
		\end{enumerate}
	\end{proof} 
	
	If $(S,K,I)$ is a split graph and $G$ is a graph disjoint from $S$, the \emph{Tyshkevich composition} $S\circ G$ of $S$ and $G$ is defined as the graph whose vertex set is $V(S\circ G)=V(S)\cup V(G)$ and whose edge set is 
	\[ E(S\circ G)=E(S)\cup E(G)\cup\{xy:x\in K,y\in G\}. \]
	This operation was introduced by R. Tyshkevich in \cite{tyshkevich2000decomposition}. 
	In general, $S\circ G$ is split if and only if $G$ is.
The next proposition, which is of independent interest, shows that balancedness behaves in a remarkably rigid way under Tyshkevich composition: whether $S_1 \circ S_2$ is balanced depends only on the second factor. Since the Tyshkevich composition is associative, it follows by iteration that a composition $S_1 \circ \cdots \circ S_k$ of split graphs is balanced if and only if its last factor $S_k$ is.
	
	\begin{proposition}
		\label{S_1S_2_balanced_iff_S_2_balanced}
		Let $S_1$ and $S_2$ be split graphs. Then, $S_1\circ S_2$ is unbalanced if and only if $S_2$ is unbalanced.
	\end{proposition}
	
	\begin{proof}
		For $i\in\{1,2\}$, let $(K_i,I_i)$ be a bipartition for $S_i$, and let $S=S_1\circ S_2$. Then, $(K,I)=(K_1\dot{\cup} K_2, I_1\dot{\cup} I_2)$ is a bipartition for $S$. In particular, $\omega(S)\ge |K_1|+|K_2|$. Moreover, for every $v\in V(S_2)$, $N_S(v)=N_{S_2}(v)\dot\cup K_1$ and $\deg_S(v)=\deg_{S_2}(v)+|K_1|$.
		
		\smallskip
		($\Rightarrow$). We prove the contrapositive. 
		Since $S_2$ is balanced, $(K_2,I_2)$ is its unique bipartition, $\omega(S_2)=|K_2|$, and $|K_2|,|I_2|\ge 2$ by \Cref{balanced_implies_|U|<|K|-1}(1). Suppose, for contradiction, that $S$ is unbalanced, and let $w\in W(S)$. By Theorem 4.4 of \cite{jaume2026nullspace}, $\deg_S(w)=\omega(S)-1$. Furthermore, by Corollary 4.8 of \cite{jaume2026nullspace}, $N_S(w)\cup w$ is a	(maximum) clique of $S$, so $N_S(w)$ is a clique of $S$. We rule out each possible location of $w$. Since $V(S)=K_1\cup K_2\cup I_1\cup I_2$, we have to analyze 4 cases.
		\begin{enumerate}[(1).]
			\item If $w\in K_1$, then $w$ is adjacent to all of $V(S_2)$, so $I_2\subseteq N_S(w)$. But $N_S(w)$ induces a clique while $I_2$ is independent with $|I_2|\ge 2$, a contradiction. Hence, $K_1\cap W(S)=\varnothing$.
			
			\item If $w\in K_2$, then $K_2-w\subseteq N_{S_2}(w)$. Since $N_{S_2}(w)\cup\{w\}$ is a clique of $S_2$, we also have $|N_{S_2}(w)|+1\le\omega(S_2)=|K_2|$. Thus, $N_{S_2}(w)=K_2-w$, and so $w$ has no neighbor in $I_2$. Hence, $(K_2-w, I_2\cup w)$ is a bipartition of $S_2$ distinct from $(K_2,I_2)$, contradicting that $S_2$ is balanced. Therefore, $K_2\cap W(S)=\varnothing$.
			
			\item If $w\in I_1$, then $N_S(w)=N_{S_1}(w)
			\subseteq K_1$ and $\deg_S(w)\le|K_1|$. But, $\deg_S(w)=\omega(S)-1\ge|K_1|+|K_2|-1\ge|K_1|+2-1$, a contradiction. Hence $I_1\cap W(S)=\varnothing$.
			
			\item If $w\in I_2$, then $N_{S_2}(w)\subseteq K_2$. Moreover, $N_{S_2}(w)\ne K_2$, for otherwise $K_2\cup w$ would be a clique of $S_2$ of size $|K_2|+1>\omega(S_2)$. Hence, $\deg_{S_2}(w)\le|K_2|-1$, so $\omega(S)-1=\deg_S(w)=\deg_{S_2}(w)+|K_1|\le|K_1|+|K_2|-1$, i.e.,	$\omega(S)\le|K_1|+|K_2|$. Together with $\omega(S)\ge|K_1|+|K_2|$ this gives
			$\omega(S)=|K_1|+|K_2|$, so $K$ is a maximum clique of $S$. As $S$ is unbalanced and	$K$ is a maximum clique (hence a clique-side of a bipartition), $K\cap W(S)\ne\varnothing$. But, $K\cap W(S)=(K_1\cup K_2)\cap W(S)=\varnothing$ by (1) and (2), a contradiction.
		\end{enumerate}
		Thus, $w\notin V(S)$, which is absurd. Hence, $S$ is balanced. 
	 
		\smallskip
		($\Leftarrow$). Since $S_2$ is unbalanced, we can assume, without loss of generality, that $|K_2|=\omega(S_2)$. Then, by Corollary 4.8 of \cite{jaume2026nullspace}, there exists a vertex $w\in K_2\cap W(S_2)$. Consequently, $(K-w, I\cup w)$ is a bipartition for $S$ different from $(K,I)$, which means that $S$ is unbalanced.
	\end{proof}

	\begin{lemma}
		\label{lemma_S-E(K)_4cycle}
		Suppose that $S = (S_1, K_1, I_1) \circ (S_2, K_2, I_2)$ is a split graph
		with $I_2 \neq \emptyset$, and let $K = K_1 \cup K_2$. If
		$\deg_S(a) \geq 2$ for some $a \in I_1$, then $S - E(K)$ contains a 4-cycle.
	\end{lemma}
	
	\begin{proof}
		Since $N_S(a) \subseteq K_1$, there exist distinct $x, y \in N_S(a) \cap K_1$.
		If $b \in I_2$, then $xb, yb \in E(S)$ by the definition of $\circ$, so
		$axbya$ is a 4-cycle in $S$. As each of its edges joins a vertex of $K$ to a
		vertex of $I_1 \cup I_2$, none of them belongs to $E(K)$. Hence, $axbya$ is
		a 4-cycle in $S - E(K)$.
	\end{proof}
	
	A graph $G$ is \emph{decomposable} if there exist graphs $S,H\neq K_0$ such that $S\circ H=G$. Otherwise, $G$ is said to be \emph{indecomposable}. The following result is a useful tool to detect whether a balanced split graph $S$ is indecomposable, provided that $S$ has certain restrictions on its degree sequence. 
	
	\begin{theorem}
		\label{indecomposable_test_split}
		Let $(S,K,I)$ be a balanced split graph such that 
		\[ \deg_S(v)\notin\{0,1,|V(S)|-1\} \]
		for all $v\in V(S)$. If $S-E(K)$ does not contain 4-cycles or does not contain 6-cycles, then $S$ is indecomposable. 
	\end{theorem}
	
	\begin{proof}
		Assume that $d_v\notin\{0,1,|V(S)|-1\}$ for all $v\in V(S)$, but $S=(S_1,K_1,I_1)\circ (S_2,K_2,I_2)$. We will show that $S-E(K)$ contains both a 4-cycle and a 6-cycle, contradicting the hypothesis. If $I_2 = \varnothing$, then $V(S_2) = K_2 \neq \varnothing$, and $N_w = K-w$ for all $w \in K_2$; hence $(K-w, I\cup w)$ would be a bipartition for $S$ different from $(K, I)$, contradicting that $S$ is balanced. Therefore, $I_2 \neq \varnothing$ (alternatively: from \Cref{S_1S_2_balanced_iff_S_2_balanced} we know that $S_2$ is balanced, which implies $I_2\neq\varnothing$ by \Cref{balanced_implies_|U|<|K|-1}(1)). If $I_1=\varnothing$, then we would have $d_v=|V(S)|-1$ for all $v\in K_1$, contradicting the hypothesis. Hence, it must be $I_1\neq\varnothing$. By hypothesis, there is a vertex $a\in I_1$ with $d_a\geq 2$. Therefore, $S-E(K)$ contains a 4-cycle, by \Cref{lemma_S-E(K)_4cycle}.
		
		It remains to exhibit a 6-cycle. Since $N_a\subseteq K_1$, there exist distinct $x,y\in N_a\cap K_1$. By \Cref{S_1S_2_balanced_iff_S_2_balanced}, $S_2$ is balanced, so $|I_2|\geq 2$ by \Cref{balanced_implies_|U|<|K|-1}(1); let $b,c\in I_2$ be distinct. If $|K_1|=2$, then every vertex of $I_1$ would be adjacent to all of $K_1$ (recall that $d_v\geq 2$ and $N_v\subseteq K_1$ for every $v\in I_1$), so every $w\in K_1$ would satisfy $N_w=V(S)-w$, i.e., $d_w=|V(S)|-1$, contradicting the hypothesis. Hence, $|K_1|\geq 3$, and we can pick $z\in K_1-\{x,y\}$. Since $b$ and $c$ are adjacent to every vertex of $K_1$ by the definition of $\circ$, $axbzcya$ is a 6-cycle in $S$. As each of its edges joins a vertex of $K$ to a vertex of $I_1\cup I_2$, none of them belongs to $E(K)$. Hence, $axbzcya$ is a 6-cycle in $S-E(K)$.
	\end{proof}
	
	The pattern in \Cref{indecomposable_test_split} does not extend beyond 6-cycles. To see this, let $S_1$ be the split graph with bipartition $(K_1,I_1)=(\{k_1,k_2,k_3\},\{a_1,a_2,a_3\})$, where $N_{S_1}(a_i)=K_1-k_i$ for each $i$. Let $S_2\approx P_4$, and consider $S=S_1\circ S_2$. Then, $S$ is a decomposable split graph of order 10, balanced by \Cref{S_1S_2_balanced_iff_S_2_balanced}, and satisfying that $\deg_S(v)\notin\{0,1,|V(S)|-1\}$ for all $v\in V(S)$. However, $S-E(K)$ contains no cycles of length greater than 6. Indeed, writing $S_2=v_1v_2v_3v_4$, we have $K_2=\{v_2,v_3\}$, and each of $v_2$, $v_3$ has exactly one neighbor in $S-E(K)$ (namely, $v_1$ and $v_4$, respectively). Hence, no cycle of $S-E(K)$ passes through $v_2$ or $v_3$, so every cycle of $S-E(K)$ lies in the bipartite subgraph with parts $K_1$ and $I_1\cup\{v_1,v_4\}$. Since $|K_1|=3$, such a cycle has length at most 6. Therefore, no analogue of \Cref{indecomposable_test_split} holds for 8-cycles or longer even cycles.
	
%
	
	We say that a graph $G$ is \emph{prime} if $G$ is active and indecomposable (in \cite{barrus.west.A4} a different, although related, family of graphs was called prime, but in this context we believe that the name fits). Note that the only indecomposable non-active graph is $K_1$. Then, prime graphs are precisely nontrivial indecomposable graphs. The next result is a useful ``primality test'' for a large class of active split graphs.
	
	\begin{corollary}
		\label{prime_test_split}
		Let $(S,K,I)$ be an active split graph without leaves. If $S-E(K)$ does not contain 4-cycles, then $S$ is prime.
	\end{corollary}
	
	\begin{proof}
		From item (7) of \Cref{split.caract.vert.activos}, we know that $S$ is balanced. Moreover, $S$ does not have isolated or universal vertices because they are inactive. Now, use \Cref{indecomposable_test_split}.
	\end{proof}

	
	\section{Active space} \label{sec:espacio.activo}
	
	The \emph{active part} of $G$, denoted by $G^*$, is the subgraph of $G$ induced by $\act(G)$. By definition, $G^*$ is active; in particular, when $\deg(G)=0$ we have $\act(G)=\varnothing$, so $G^*=K_0$, which is why $K_0$ is taken to be active. Thus, $G$ is active if and only if $G^*=G$. Since every 2-switch of $G$ acts only on active vertices, $\deg(G^*)=\deg(G)$. If $d=d(G)$, we refer to the degree sequence of $G^*$ as $d^*$.
	
	\begin{lemma}
		\label{tau(G-S)=tau(G)-S}
		Let $G$ be a graph and let $V_{\theta}$ be the set of all vertices activated by some 2-switch in the sequence $\theta =(\tau_i )_{i=1}^k$ transforming $G$. If $W\subseteq V(G)$ and $W\cap V_{\theta}=\varnothing$, then
		\[ \theta(G-W)=\theta(G)-W. \]
	\end{lemma}
	
	\begin{proof}
		Let $k=1$, that is, $\theta$ is a single 2-switch ${{a \ b}\choose{c \ d}}$ on $G$. Since $W\cap V_{\theta}=\varnothing$, it is clear that $\theta$ is also a 2-switch on $G-W$. Then,
		\[ \theta(G-W)=((G-W)-\{ab,cd\})+\{ac,bd\}. \]
		Since $V_{\theta}=\{a,b,c,d\}$ and $W\cap V_{\theta}=\varnothing$, it follows that
		\[ (G-W)-\{ab,cd\}=(G-\{ab,cd\})-W. \]
		For the same reason, it is easy to see that
		\[ ((G-\{ab,cd\})-W)+\{ac,bd\}= \]
		\[ ((G-\{ab,cd\})+\{ac,bd\})-W=\theta(G)-W. \]
		The rest of the proof is carried out easily by induction on $k$, since it suffices to recycle the case $k=1$ by replacing $G$ with $\tau_1\ldots\tau_{k-1}(G)$ and $\theta$ with $\tau_k$.
	\end{proof}
	
	Let $X=(V,E)$ be a graph. For any vertex set $V'$ disjoint from $V$ and any 
	\[ E'\subseteq\binom{V\cup V'}{2}-\binom{V}{2}=\binom{V'}{2}\dot{\cup}\{xy:x\in V,y\in V'\}, \]
	we define the $(V',E')$\emph{-extension} of $X$ as the graph 
	\[ \ext(X,V',E')=(V\cup V',E\cup E'). \]
	Clearly, $X\preceq \ext(X,V',E')$. 
	
	\begin{lemma}
		\label{ext(theta)=theta(ext)}
		Let $G=(V,E)$ be a graph and let $\theta$ be a sequence of 2-switches transforming $G$. If $W=V-\act(G)$ and $E'=\{xy\in E:x\in W\}$, then
		\[ \ext(\theta(G^*),W,E')=\theta(\ext(G^*,W,E')). \]
	\end{lemma}
	
	\begin{proof}
		The first thing to note is that $G^*=G-W$. Hence, it is clear that $\ext(G^*,W,E')=G$. Secondly, since each 2-switch in $\theta$ only deletes and adds edges of the form $xy$, with $x,y\in \act(G)$, it follows that
		\[ E'=\{xy\in E(\theta(G)):x\in W\} \]
		(remember that $\act(G)$ is an invariant associated with $d(G)$, by \Cref{2switch.preservs.act}). With this in mind, and applying \Cref{tau(G-S)=tau(G)-S}, we obtain
		\[ \ext(\theta(G^*),W,E')=\ext(\theta(G)-W,W,E')= \]
		\[ (V,E(\theta(G)-W)\dot{\cup}E'). \]
		Since
		\[ E(\theta(G)-W)=E(\theta(G))-\{xy\in E(\theta(G)): x\in W\}, \]
		we conclude that $\ext(\theta(G^*),W,E')=\theta(G)$.
	\end{proof}
	
	If $\mathcal{X}\preceq\mathcal{G}(d)$, define
	\[ V(\mathcal{X})^* =\{G^* :G\in V(\mathcal{X})\}. \]
	Let $\psi :V(\mathcal{X})\rightarrow V(\mathcal{X})^*$ be the function defined by $\psi(X)= X^*$. Choose arbitrarily a vertex $X_0$ of $\mathcal{G}(d)$ and define
	\[ E_0=\{xy\in E(X_0):x\in W\}, \]
	where $W=V(X_0)-\act(d)$. Thanks to \Cref{2switch.preservs.act}, it is easy to verify that $E_0$ is an invariant associated with $d$. In other words, replacing $X_0$ with any other vertex of $\mathcal{G}(d)$, the set $E_0$ does not change. Now suppose that $G^* =H^*$. Then,
	\[ G=\ext(G^*,W,E_0)=\ext(H^*,W,E_0)=H, \]
	which shows that $\psi$ is injective. Since $V(\mathcal{X})^*=\psi(V(\mathcal{X}))$, $\psi$ is obviously surjective. Therefore, $\psi$ is a bijection. 
	
	Let $G^*,H^*\in V(\mathcal{X})^*$, $G^*\neq H^*$. According to \Cref{berge's.theorem}, there exists a sequence $\theta$ of 2-switches such that $H=\theta(G)$, because $G,H\in V(\mathcal{G}(d))$ and $G\neq H$. Then, applying \Cref{2switch.preservs.act} and \Cref{tau(G-S)=tau(G)-S} for $W=V(G)-\act(d)$, we deduce that 
	\[ \theta(G^*)=\theta(G-W)=\theta(G)-W=H-W=H^*, \]
	which means that $d(G^*)=d^*=d(H^*)$. In other words, $V(\mathcal{X})^*\subseteq V(\mathcal{G}(d^*))$. We define the \emph{active space} $\mathcal{X}^*$ associated with $\mathcal{X}\preceq\mathcal{G}(d)$ as the subgraph of $\mathcal{G}(d^*)$ induced by $V(\mathcal{X})^*$. Thanks to \Cref{tau(G-S)=tau(G)-S} and \Cref{2switch.preservs.act}, it is clear that $\psi$ is a homomorphism between $\mathcal{X}$ and $\mathcal{X}^*$. On the other hand, observe that the function $\zeta: V(\mathcal{X}^*)\rightarrow V(\mathcal{X})$, defined by $\zeta(Y)=\ext(Y,W,E_0)$, is the inverse of $\psi$, and it is also a homomorphism by \Cref{ext(theta)=theta(ext)}. Then, $\psi$ is an isomorphism. 
	
	\begin{theorem}
		\label{isomorfismo.espacio.activo}
		If $\mathcal{X}\preceq\mathcal{G}(d)$, then $\mathcal{X}\approx\mathcal{X}^*$.
	\end{theorem}
	
	\begin{proof} 
		It follows from the previous discussion.
	\end{proof}
	
	\begin{corollary}
		\label{G(d)_iso_G(d*)}
		If $d$ is a graphical sequence, then $\mathcal{G}(d)\approx\mathcal{G}(d^*)$.
	\end{corollary}
	
	\begin{proof}
		By \Cref{isomorfismo.espacio.activo} with $\mathcal{X}=\mathcal{G}(d)$, we have $\mathcal{G}(d)\approx\mathcal{G}(d)^*$, the latter being the subgraph of $\mathcal{G}(d^*)$ induced by $V(\mathcal{G}(d))^*$. Hence, it suffices to show that $V(\mathcal{G}(d))^*=V(\mathcal{G}(d^*))$. The inclusion $\subseteq$ was already established; for $\supseteq$, let $Y\in V(\mathcal{G}(d^*))$. Since $d(Y)=d^*=d(X_0^*)$, \Cref{berge's.theorem} gives a sequence $\theta$ of 2-switches with $Y=\theta(X_0^*)$. Each 2-switch of $\theta$ acts on four vertices of $\act(d)$, so $W\cap V_\theta=\varnothing$; since $X_0^*=X_0-W$ is an induced subgraph of $X_0$, the sequence $\theta$ is also applicable to $X_0$ (by iterating \Cref{tau(G-S)=tau(G)-S}). Thus $G=\theta(X_0)\in V(\mathcal{G}(d))$ and, by \Cref{tau(G-S)=tau(G)-S},
		\[ G^*=\theta(X_0)-W=\theta(X_0-W)=\theta(X_0^*)=Y. \]
		Therefore, $Y\in V(\mathcal{G}(d))^*$.
	\end{proof}
	
	
	\section{Basic properties of the degree} \label{sec:prop.basicas.deg}
	
	At this point, it should come as no surprise that the number of 2-switches on a graph $G$ can be obtained by summing the degree of all induced subgraphs of order 4 in $G$. In other words,
	\begin{equation*}
		\deg(G)=\sum_{H\in Q_G}\deg(H),
	\end{equation*}
	where $\deg(H)=2$ if $H\approx 2K_2$ or $C_4$, $\deg(H)=1$ if $H\approx P_4$, and $\deg(H)=0$ in the remaining cases. If $X$ is one of the 11 graphs in \Cref{los.11.de.orden4}, we define
	\[ Q_G (X)=\{H\in Q_G :H\approx X\}. \]
	Then, we have the following result.
	
	\begin{theorem}
		\label{degreeofG}
		For every graph $G$,
		\begin{equation}
			\label{eq18}
			\deg(G)=2|Q_G (2K_2 )|+2|Q_G (C_4 )|+|Q_G (P_4 )|.
		\end{equation}
		In particular, if $G$ is a split graph, then $\deg(G)=|Q_G (P_4 )|$.
	\end{theorem}
	
	\begin{proof}
		Equality \eqref{eq18} follows from the previous discussion. When $G$ is split, we know it does not contain induced subgraphs isomorphic to $C_4$ or $2K_2$. Therefore, $|Q_G (2K_2 )|=|Q_G (C_4 )|=0$.
	\end{proof}
	
	From \eqref{eq18} we observe that
	\[ \deg(G)\equiv|Q_G(P_4)|\pmod{2}, \]
	that is, the degree of $G$ has the same parity as the number of induced $P_4$'s in $G$. In particular, every graph of odd degree contains at least one induced $P_4$.
	
	\begin{corollary}
		\label{degG=deg(G.complemento)}
		For any graph $G$,
		\begin{equation*}
			\deg(G)=\deg(\overline{G}).
		\end{equation*}
	\end{corollary}
	
	\begin{proof}
		If $G\in V(\mathcal{G}(d))$, we know that $\overline{G}\in V(\mathcal{G}(\overline{d}))$. By \Cref{dual.spaces.iso}, we have $\mathcal{G}(d)\approx\mathcal{G}(\overline{d})$ via the isomorphism $G\mapsto\overline{G}$, and therefore $\deg(G)=\deg(\overline{G})$.
	\end{proof}
	
	It is worth mentioning that \Cref{degG=deg(G.complemento)} can also be proven using the fact that $H\preceq G$ if and only if $\overline{H}\preceq\overline{G}$. This, together with the fact that $\overline{P_4}=P_4$ and $\overline{C_4}=2K_2$, implies that $|Q_G(P_4)|=|Q_{\overline{G}}(P_4)|$, $|Q_G(C_4)|=|Q_{\overline{G}}(2K_2)|$, and $|Q_G(2K_2)|=|Q_{\overline{G}}(C_4)|$.
	
	\begin{proposition}
		\label{deg.respects.ind.inc}
		If $H\preceq G$, then 
		\begin{equation*}
			\deg(H)\leq\deg(G).
		\end{equation*}
	\end{proposition}
	
	\begin{proof}
		This is true because $Q^*_H \subseteq Q^*_G$.
	\end{proof}
	
	\Cref{deg.respects.ind.inc} is not generally true if the subgraph is not induced. For example, $P_4 \subseteq K_4$ but $\deg(P_4)=1$ and $\deg(K_4)=0$. \Cref{deg.respects.ind.inc} can be easily generalized as follows.
	
	\begin{proposition}
		\label{collH.ind.sub.G.degG>=sum.degH}
		Let $\{ H_i :i\in[k]\}$ be a collection of induced subgraphs of $G$ such that $|E(H_i )\cap E(H_j )|\leq 1$ for $i\neq j$. Then, 
		\begin{equation*}
			\sum_{i=1}^{k}\deg(H_i )\leq\deg(G).
		\end{equation*}
	\end{proposition}
	
	\begin{proof}
		If $\tau$ were a 2-switch on both $H_i$ and $H_j$ for certain distinct $i,j$, then we would have $|E(H_i )\cap E(H_j )|\geq 2$, contradicting the hypothesis. Therefore, each 2-switch on $H_i$ cannot be performed on $H_j$, for every $j\neq i$.
	\end{proof}
	
	\begin{proposition}
		If $G'\preceq G^{*}$ and $\deg(G')=\deg(G)>0$, then $G'=G^{*}$. In other words, $G^{*}$ is the smallest induced subgraph (in terms of order) of $G$ with the same degree as $G$.
	\end{proposition}
	
	\begin{proof}
		Suppose that $G'\preceq G^{*}$ and $\deg(G')=\deg(G)>0$, but $G'\neq G^{*}$. Since $\deg(G)=\deg(G^*)>0$, we have $\act(G^*)\neq\varnothing$. Then, if $v\in V(G^{*})-V(G')$, there is at least one 2-switch $\tau$ on $G^*$ activating $v$. Since $v\notin V(G')$, it is clear that $\tau$ cannot be applied to $G'$. Using \Cref{deg.respects.ind.inc}, we finally obtain
		\[ \deg(G)=\deg(G')<\deg(G^*)=\deg(G), \]
		which is absurd.
	\end{proof}
	
	The next result states that computing the degree of a graph essentially reduces to computing the degree of its components.
	
	\begin{proposition}
		\label{degree.disconn.G}
		Let $G$ be a graph with $k\geq 2$ connected components. If $G_{i}$ is a component of $G$, then 
		\begin{equation}
			\deg(G)=\sum_{i=1}^{k}\deg(G_{i}) +\sum_{1\leq i<j\leq k}2|E (G_{i})||E(G_{j})|.	
		\end{equation}	
	\end{proposition}
	
	\begin{proof}
		Suppose $k=2$. Since $G_1$ and $G_2$ are disjoint induced subgraphs of $G$, the degree of $G$ is at least $\deg(G_1)+\deg(G_2)$, by \Cref{collH.ind.sub.G.degG>=sum.degH}. The remaining 2-switches that transform $G$ are those that replace edges from different components. More specifically: if $ab\in G_1$ and $cd\in G_2$, then ${{a \ b}\choose{c \ d}}$ and ${{a \ b}\choose{d \ c}}$ are 2-switches on $G$. Therefore, 
		\[ \deg(G)=\deg(G_1)+\deg(G_2) +2|E(G_1)||E(G_2)|. \]
		The rest of the proof follows easily by induction on $k$.
	\end{proof}
	
As we have just seen in \Cref{degree.disconn.G}, the degree of a graph is not additive with respect to its connected components (i.e., $\deg(G_1\dot{\cup}G_2)\neq \deg(G_1)+\deg(G_2)$ in general). Nevertheless, the next result shows that the quantity $\delta(G)=|E(G)|^2-\deg(G)$ is additive in that sense.

\begin{proposition}
	\label{delta.additive.disjoint.union}
		Let $G$ be a graph. If $G_1,\ldots,G_k$ are the connected components of $G$, then 
	\begin{equation*}
		\delta(G)=\sum_{i=1}^{k}\delta(G_i).
	\end{equation*}
\end{proposition}

\begin{proof}
	It suffices to prove the case $k=2$; the general statement then follows by a straightforward induction on $k$. So, let $G=G_1\,\dot\cup\,G_2$. By \Cref{degree.disconn.G},
	\[ \deg(G)=\deg(G_1)+\deg(G_2)+2|E(G_1)||E(G_2)|. \]
	Since $|E(G)|=|E(G_1)|+|E(G_2)|$, expanding the square yields
	\[ |E(G)|^2=|E(G_1)|^2+|E(G_2)|^2+2|E(G_1)||E(G_2)|. \]
	Subtracting, the cross term cancels and we obtain
	\[ |E(G)|^2-\deg(G)=\big(|E(G_1)|^2-\deg(G_1)\big)+\big(|E(G_2)|^2-\deg(G_2)\big). \qedhere \]
\end{proof}	
	
\begin{lemma}
	\label{no.A4.cruzado}
	Let $(S,K,I)$ be a split graph and $G$ a graph. Then, no induced $2K_2$, $P_4$, or $C_4$ of $S\circ G$ meets both $V(S)$ and $V(G)$; equivalently, $Q_{S\circ G}^*=Q_S^*\,\dot\cup\,Q_G^*$.
\end{lemma}

\begin{proof}
	The inclusion $Q_S^*\cup Q_G^*\subseteq Q_{S\circ G}^*$ is clear. Since $V(S)\cap V(G)=\varnothing$, we also have $Q_S^*\cap Q_G^*=\varnothing$. Suppose there is an $H\in Q_{S\circ G}^*$ such that $V(H)$ intersects both $V(S)$ and $V(G)$.
	
	If $V(H)\cap K=\varnothing$, pick $u\in V(H)\cap I$. Then, $u$ is nonadjacent to $V(H)\cap I$ (as $I$ is independent) and to $V(H)\cap V(G)$ (no edges between $I$ and $V(G)$), so $u$ is isolated in $H$, a contradiction.
	
	If $V(H)\cap K\neq\varnothing$, then $|V(H)\cap K|=1$; otherwise, two vertices of $V(H)\cap K$ with any vertex of $V(H)\cap V(G)$ would induce a triangle. Write $V(H)\cap K=\{x\}$. Each vertex of $V(H)\cap I$ has $x$ as its only possible neighbor in $H$, so $V(H)\cap I\subseteq N_H(x)$; since also $V(H)\cap V(G)\subseteq N_H(x)$, we get $\deg_H(x)=3$, a contradiction.
\end{proof}	

\begin{theorem}
	\label{deg(SoG)=deg(S)+deg(G)}
	If $(S,K,I)$ is a split graph and $G$ is a graph, then 
	\begin{equation*}
		\label{eq17}
		\deg(S\circ G)=\deg(S)+\deg(G).
	\end{equation*}
\end{theorem}

\begin{proof}
	By \Cref{degreeofG} and \Cref{no.A4.cruzado}, $\deg(S\circ G)=\sum_{H\in Q_{S\circ G}^*}\deg(H)=\sum_{H\in Q_S^*}\deg(H)+\sum_{H\in Q_G^*}\deg(H)=\deg(S)+\deg(G)$.
\end{proof}

Note that, although the graph $S\circ G$ may depend on the chosen bipartition $(K,I)$ of $S$ when $S$ is unbalanced, its degree does not: by \Cref{deg(SoG)=deg(S)+deg(G)}, $\deg(S\circ G)=\deg(S)+\deg(G)$ regardless of the bipartition. In particular, taking $S=K_n$ gives $\deg(K_n\circ G)=\deg(G)$, which shows that the degree of a graph imposes no bound on its order or size.

\begin{theorem}
	\label{SoG.act.iff.S_G.act}
	$S\circ G$ is active if and only if $S$ and $G$ are active.
\end{theorem}

\begin{proof}
	By \Cref{2switch.preservs.act}, a vertex is active in a graph if and only if it is non-isolated in its $A_4$-structure. By \Cref{no.A4.cruzado}, no induced $2K_2$, $P_4$, or $C_4$ of $S\circ G$ meets both $V(S)$ and $V(G)$, and so we have $A_4(S\circ G)=A_4(S)\,\dot\cup\,A_4(G)$. Hence, a vertex of $S\circ G$ is non-isolated in $A_4(S\circ G)$ if and only if it is non-isolated in the $A_4$-structure of the factor containing it. Therefore, $S\circ G$ is active if and only if both $S$ and $G$ are.
\end{proof}

	\section{Explicit formulas} \label{sec:formulas.explicitas}
	
	If $G$ is a graph, we define
	\[ \dpe(G)=|\{ \{e,f\}: e,f\in E(G), e\cap f=\varnothing \}|. \]
	In other words, $\dpe(G)$ is the number of unordered pairs of disjoint edges in $G$. The abbreviation “$\dpe$” stands for “disjoint pairs of edges”.
	
	\begin{proposition}
		\label{dpe.formula}
		If $G$ is a graph with degree sequence $d=(d_{v})_{v=1}^{n}$, then:
		\begin{equation*}
			\dpe(G)=\binom{|E(G)|}{2}-\sum_{v=1}^{n} \binom{d_{v}}{2}=\binom{|E(G)| +1}{2}-\frac{1}{2}|d|^{2},
		\end{equation*}
		where $|d|^2=d_1^2+\ldots+d_n^2$ and $\binom{i}{j}=0$ if $i<j$. In particular, $\dpe(G)$ is constant on $V(\mathcal{G}(d))$, i.e., it is an invariant associated with $d$.
	\end{proposition}
	
	\begin{proof}
		We can obtain $\dpe(G)$ by computing the number of pairs of non-disjoint edges in $G$ and then subtracting this number from $\binom{|E(G)|}{2}$, i.e., the total number of edge pairs in $G$. Since there are $d_{v}$ edges incident to each $v\in V(G)$, we get $\binom{d_{v}}{2}$ pairs of edges sharing $v$, and hence $\sum_{v=1}^{n} \binom{d_{v}}{2}$ edge pairs with a common vertex.
		
		The second equality follows from expanding the binomial coefficients and applying simple manipulations.
	\end{proof}
	
	Motivated by \Cref{dpe.formula}, we may use the notation $\dpe(d)$ instead of $\dpe(G)$ whenever $d=d(G)$. Moreover, note that 
	\[ \dpe(G)=|\{H\subseteq G:H\approx 2K_2 \}|. \]
	The next proposition gives a formula expressing the $\dpe$ of a graph in terms of the $\dpe$ of its connected components.
	
	\begin{proposition}
		\label{dpe.components}
		If $G$ is a graph with $k$ components $G_i$, then
		\begin{equation*}
			\dpe(G)=\sum_{i=1}^{k}\dpe(G_i )+\sum_{1\leq i<j\leq k}|E(G_i)||E(G_j)|.
		\end{equation*}
	\end{proposition}
	
	\begin{proof}
		Let $k=2$, i.e., $G=G_1\dot{\cup}G_2$. If $e_1\in E(G_1)$ and $e_2\in E(G_2)$, clearly $e_1\cap e_2=\varnothing$ and there are $|E(G_2)|$ pairs of the form $\{e,e_2\}$ for each $e\in E(G_1)$. Then,
		\[ \dpe(G)=\dpe(G_1)+\dpe(G_2)+|E(G_1)||E(G_2)|. \]
		The rest follows easily by induction on $k$.
	\end{proof}
	
	\begin{proposition}
		\label{deg<=m^2-m}
		For any graph $G$ of order $n$ and size $m\geq 0$,
		\begin{equation*}
			\deg(G)\leq 2\dpe(G)\leq m(m-1),
		\end{equation*}
		and equality holds if and only if $G\approx (mK_2) \dot{\cup}\overline{{K}_{n-2m}}$.
	\end{proposition}
	
	\begin{proof}
		For each pair of disjoint edges in $G$, there are at most two 2-switches that replace them. Hence,
		\begin{equation*}
			\deg(G)\leq 2\dpe(G)=m(m-1)-\sum_{v=1}^{n}(d_v^{2}-d_v )\leq m(m-1).
		\end{equation*}
		Clearly, equality is achieved precisely when $d_v \in\{0,1\}$ for each vertex $v\in G$. Under this condition, we must have $d_u =d_v =1$ for every $uv\in E(G)$, and the remaining $n-2m$ vertices of $G$ are isolated.
	\end{proof}
	
	From \Cref{deg<=m^2-m} it follows that $\delta(G)\ge |E(G)|$, for any graph $G$ (recall \Cref{delta.additive.disjoint.union}). 
	
	\begin{proposition}
		\label{number.of.P4.in.G}
		
		Let $G$ be a graph. If $p_4 (G)=|\{H\subseteq G: H\approx P_4\}|$ and $k_3 (G)=|\{H\subseteq G: H\approx K_3\}|$, then
		\begin{equation}
			\label{eq43}
			p_4 (G)+3k_3 (G)=\sum_{uv\in E(G)}(d_{u}-1)(d_{v}-1) .	
		\end{equation}
	\end{proposition}
	
	\begin{proof}
		For each edge $uv\in G$, we have $d_{u}-1$ options to attach an edge $ux\neq uv$ at $u$, and $d_{v}-1$ options to attach an edge $vy\neq uv$ at $v$. If $G$ has no triangles, there are exactly $(d_{u}-1)(d_{v}-1)$ different $P_4$'s for each edge $uv$ in $G$. If $G$ has triangles, then
		\begin{equation}
			\label{eq44}
			(d_u -1)(d_v -1) = |P_4^*(uv)|+|K_3(uv)|,
		\end{equation} 
		where $P_4^*(uv)$ is the set of all $P_4$'s in $G$ with $uv$ as the central edge, and $K_3(uv)$ is the set of all triangles in $G$ containing $uv$. Since
		\[ \{H\subseteq G: H\approx P_4\} = \dot{\bigcup}_{e\in E(G)} P_4^*(e), \]
		it is clear that $\sum_{e\in E(G)} |P_4^*(e)|=p_4(G)$. On the other hand, let $M$ be the multiset given by the union, over $e\in E(G)$, of all $K_3(e)$. Since each triangle in $\{H\subseteq G:H\approx K_3\}$ appears in $M$ with multiplicity 3, it follows that $\sum_{e\in E(G)} |K_3(e)|=3k_3(G)$. Finally, \eqref{eq43} follows from \eqref{eq44} by summing over all $uv\in E(G)$.
	\end{proof}
	
	We denote by $c_4(G)$ and $k_4(G)$ the number of cycles and cliques of order 4 in $G$, respectively. More precisely:
	\[ c_4(G) = |\{H \subseteq G : H \approx C_4\}|, \]
	\[ k_4(G) = |\{H \subseteq G : H \approx K_4\}| = |Q_G(K_4)|. \]
	
	\begin{theorem}
		\label{deg.general.formula}
		If $G\in V(\mathcal{G}(d))$, then
		\begin{equation}
			\label{deg.formula}
			\deg(G) = 2\dpe(d) + 2c_4(G) - p_4(G).
		\end{equation}
		Letting $|E(G)| = m$, \eqref{deg.formula} can be rewritten as:
		\begin{equation*}
			\deg(G) = m(m+1) - |d|^2 + 2c_4(G) + 3k_3(G) - \sum_{uv \in E(G)} (d_u - 1)(d_v - 1).
		\end{equation*}
	\end{theorem}
	
\begin{proof}
	First, recall that $\dpe(G) = |\{H \subseteq G : H \approx 2K_2\}|$. If $Y \preceq G$ and $|Y| = 4$, observe that
	\begin{equation}
		\label{contar.subgrafos}
		|\{H \subseteq G : H \approx Y\}| = \sum_{i=1}^{11} |Q_G(X_i)| \cdot |\{H \subseteq X_i : H \approx Y\}|,
	\end{equation}
	where the sum in \eqref{contar.subgrafos} ranges over the 11 graphs $X_i$ in \Cref{los.11.de.orden4}. To simplify notation, we write $|X|$ instead of $|Q_G(X)|$ and use $p_4, c_4, k_4, \dpe$ instead of $p_4(G), c_4(G), k_4(G), \dpe(d)$. Applying formula \eqref{contar.subgrafos} for $Y \in \{C_4, P_4, 2K_2\}$, we obtain:
	\begin{align*}
		c_4 &= |D_4| + |C_4| + 3k_4, \\
		p_4 &= 4|C_4| + |P_4| + 2|U_4| + 6|D_4| + 12k_4, \nonumber\\
		\dpe &= 2|C_4| + |P_4| + |2K_2| + |U_4| + 2|D_4| + 3k_4. \nonumber
	\end{align*}
	Subtracting the second equation from twice the third yields
	\[ 2\dpe - p_4 = |P_4| + 2|2K_2| - 2|D_4| - 6k_4, \]
	that is,
	\[ |P_4| + 2|2K_2| = 2\dpe - p_4 + 2\big(|D_4| + 3k_4\big). \]
	Using that $|D_4| + 3k_4 = c_4 - |C_4|$, we obtain
	\[ |P_4| + 2|2K_2| + 2|C_4| = 2\dpe + 2c_4 - p_4. \]
	Finally, the left-hand side equals $\deg(G)$, by \Cref{degreeofG}.
\end{proof}	
	
	
	\Cref{deg.general.formula} says that computing the 2-switch-degree of $G$ essentially reduces to determining $c_4(G)$ and $k_3(G)$ (on which $p_4(G)$ depends). It is well known that these can be obtained from the \emph{adjacency matrix} $A$ of $G$, an $n \times n$ matrix (where $n = |V(G)|$), with entry $[A]_{ij} = 1$ if $ij \in E(G)$ and $0$ otherwise. In fact:
	\[ k_3(G) = \frac{1}{6} \, \text{tr}(A^3), \quad c_4(G) = \frac{1}{2} \sum_{1 \leq i < j \leq n} \binom{[A^2]_{ij}}{2}, \]
	where $\text{tr}(\ast)$ is the trace. We also have:
	\begin{equation*}
		\text{tr}(A^4) = 8c_4(G) + 2|E(G)| + 4\sum_{v=1}^{n} \binom{d_v}{2} = 8c_4(G) - 2|E(G)| + 2|d|^2,
	\end{equation*}
	\begin{equation*}
		c_4(G) = \frac{1}{8} \text{tr}(A^4) + \frac{1}{4} |E(G)| - \frac{1}{4} |d|^2 = \frac{1}{8} \text{tr}(A^4) + \frac{1}{2} \dpe(d) - \frac{1}{4} |E(G)|^2.
	\end{equation*}
	Since $\dpe(d)$ and $\sum_{uv\in E(G)}(d_u-1)(d_v-1)$ are obtained from $d=d(G)$ in $O(n^2)$ time, and both $k_3(G)=\frac{1}{6}\operatorname{tr}(A^3)$ (via $\operatorname{tr}(A^3)=\sum_{i,j}[A^2]_{ij}[A]_{ij}$) and $c_4(G)$ can be read off from the single matrix product $A^2$ in $O(n^2)$ additional time, the computation of $\deg(G)$ reduces to one matrix squaring. Consequently, $\deg(G)$ is computable in $O(n^3)$ time (indeed, within the cost of a single $n\times n$ matrix multiplication) for every graph.
	
	
	As immediate corollaries of \Cref{deg.general.formula}, we can derive simplified formulas for specific graph classes. For example, if $G$ is triangle-free (i.e., $G$ contains no 3-cycles), then $k_3(G) = 0$. If $G$ is $C_4$-free, then $c_4(G) = 0$. The \emph{girth} of a graph $G$, denoted by $g(G)$, is the length of the shortest cycle contained in $G$. If $G$ is acyclic, $g(G)$ is defined as $\infty$. Note that $g(G) \geq 5$ if and only if $k_3(G) = 0 = c_4(G)$.
	
	\begin{corollary}
		\label{deg.girth>4}
		If $G \in V(\mathcal{G}(d))$ and $g(G) \geq 5$, then
		\begin{equation}
			\label{deg.girth>4.formula}
			\deg(G) = 2\dpe(d) - p_4(G).
		\end{equation}
		Letting $|E(G)| = m$, formula \eqref{deg.girth>4.formula} can be rewritten as
		\begin{equation*}
			\deg(G) = m(m+1) - |d|^2 - \sum_{uv \in E(G)} (d_u - 1)(d_v - 1).
		\end{equation*}
	\end{corollary}
	
	\begin{proof}
		It follows from the previous discussion.
	\end{proof}
	
	As an application of \Cref{deg.girth>4}, we can compute the degree of paths and cycles, since $g(P_n) = \infty$ and $g(C_n) = n$.
	
	Paths of order 1 or 2 clearly have degree 0. For $n \geq 3$, $d=d(P_n) = 2^{n-2}1^2$. Since $\dpe(d) = \binom{n-2}{2}$ and $p_4(P_n) = n - 3$, we get
	\begin{equation}
		\label{path.deg}
		\deg(P_n) = (n - 3)^2.
	\end{equation}
	For each $n \geq 3$, $d=d(C_n) = 2^n$ and $\dpe(d) = \binom{n}{2} - n$. Clearly, $\deg(C_3) = 0$ and $\deg(C_4) = 2$. If $n \geq 5$, then $p_4(C_n) = n$ and therefore,
	\begin{equation}
		\label{ciclo.deg}
		\deg(C_n) = n(n - 4).
	\end{equation}
	Note that $\deg(C_n) \geq \deg(P_n)$ for all $n \geq 3$. We can apply formulas \eqref{path.deg} and \eqref{ciclo.deg} to obtain a lower bound on the degree of a graph $G$ that contains some path $P$ or cycle $C$ as an induced subgraph. Indeed, from \Cref{deg.respects.ind.inc} we have $\deg(G) \geq \deg(P), \deg(C)$.
	
	Let $G$ be a graph with degree sequence $d = (d_v)_{v=1}^n$. The numbers
	\[ \zeta_{1}(G)=\sum_{v=1}^n d_{v}^{2}=|d|^{2}, \quad  \zeta_{2}(G)=\sum_{uv\in E(G)} d_{u}d_{v} \]
	are called the \emph{first} and \emph{second Zagreb indices} of $G$, respectively. These are well-known parameters in Chemical Graph Theory and widely studied in the literature. These so-called topological invariants are often used to quantify structural characteristics of chemical compounds. Both $\zeta_1$ and $\zeta_2$ were first introduced in 1972 by Gutman and Trinajstić (see \cite{gutman1972graph}) and have been used to relate molecular properties such as chemical stability and $\pi$-electron energy. The latter refers to the electrons involved in $\pi$ bonds, a type of covalent bond formed by the lateral overlap of $\pi$ atomic orbitals. There is a relationship between the total energy of $\pi$ electrons and the Zagreb indices, since the latter capture aspects of molecular connectivity and branching that influence electron delocalization. For example, in certain studies, it has been found that greater branching (reflected in higher Zagreb index values) can lead to lower stability and thus higher total $\pi$-electron energy. Therefore, Zagreb indices can be useful tools for estimating and understanding how molecular structure affects $\pi$-electron energy.
	
	It is a curious fact that these chemical indices arise naturally in our context. Indeed, since $\sum_{uv \in E(G)} (d_u + d_v) = \zeta_1(G)$, we have
	\begin{equation}
		\label{sum.zagreb.1,2}
		\sum_{uv \in E(G)} (d_u - 1)(d_v - 1) = \zeta_2(G) - \zeta_1(G) + |E(G)|.
	\end{equation}
	We know that the left-hand sum in \eqref{sum.zagreb.1,2} (which equals $p_4(G) + 3k_3(G)$ by \Cref{number.of.P4.in.G}) is part of the formula in \Cref{deg.general.formula} and \Cref{deg.girth>4} for computing $\deg(G)$. Hence, we see a link between the 2-switch and the Zagreb indices. The next theorem tells us how the 2-switch-degree is tied to $\zeta_2$.
	
	\begin{theorem}
		\label{relacion.deg.zeta2}
		If $G$ is a graph, then
		\[ \deg(G) + \zeta_2(G) = |E(G)|^2 + 2c_4(G) + 3k_3(G). \]
	\end{theorem}
	
	\begin{proof}
		It follows from the previous discussion.
	\end{proof}
	
	When $g(G) \geq 5$, we obtain a remarkable corollary of \Cref{relacion.deg.zeta2}: the sum of the 2-switch-degree and the second Zagreb index equals the square of the size.
	
	\begin{corollary}
		If $G$ is a graph with $g(G) \geq 5$, then
		\begin{equation}
			\label{eq.deg.zeta2.girth>4}
			\deg(G) + \zeta_2(G) = |E(G)|^2.
		\end{equation}
	\end{corollary}
	
	\begin{proof}
		It follows immediately from \Cref{relacion.deg.zeta2}, since $c_4(G) = 0 = k_3(G)$ when $g(G) \geq 5$.
	\end{proof}
	
	Now suppose we are studying extremal values of the second Zagreb index (and which graphs attain them) in a family of graphs with fixed size. Equality \eqref{eq.deg.zeta2.girth>4} may be very useful in this case, as it tells us that minimizing (maximizing) $\zeta_2$ is equivalent to maximizing (minimizing) the degree.
	
	
	\section{Degree of trees and unicyclic graphs} \label{sec:deg.arboles.unicicl}
	
	We denote by $\mathcal{F}(d)$ the subgraph of $\mathcal{G}(d)$ induced by vertices of $\mathcal{G}(d)$ that are forests. Since all members of $V(\mathcal{F}(d))$ have the same number of connected components, it follows that if one of them is a tree, then all the others are trees as well. A 2-switch $\tau$ on a forest $F \in V(\mathcal{F}(d))$ is called an \emph{f-switch} on $F$ if $\tau(F) \in V(\mathcal{F}(d))$. When $F$ is a tree, we use the term \emph{t-switch} instead of f-switch and the symbol $\mathcal{T}(d)$ instead of $\mathcal{F}(d)$. The \emph{f-degree} of a forest $F$, denoted by $\deg_f(F)$, is the degree of the vertex $F$ in the graph $\mathcal{F}(d)$. Equivalently, $\deg_f(F)$ is the number of f-switches that act on $F$. Clearly, $\deg_f(F) \leq \deg(F)$, because $\mathcal{F}(d) \preceq \mathcal{G}(d)$.
	
	The following result provides a formula to compute the f-degree of a tree.
	
	\begin{theorem}
		\label{deg_f.tree}
		Let $T$ be a tree with $d(T)=d=(d_{v})_{v=1}^{n}$. Then,
		\begin{equation*}
			\deg_f(T) = \dpe(d) = \binom{n-1}{2} - \sum_{v=1}^{n} \binom{d_v}{2} = \binom{n}{2} - \frac{1}{2}|d|^2,
		\end{equation*}
		where $\binom{i}{j}=0$ if $i<j$.
	\end{theorem}
	
	\begin{proof}
		Assuming that $xb\ldots cy$ is the path connecting $x$ to $y$ in $T$, we observe that:
		\begin{enumerate}[(1).]
			\item $\binom{x \ b}{c \ y}$ is a t-switch on $T$;
			\item if $bc \in E(T)$, then $\binom{x \ b}{y \ c}$ cannot be applied to $T$;
			\item if $bc \notin E(T)$, then $\binom{x \ b}{y \ c}$ disconnects $T$, and so is not a t-switch on $T$.
		\end{enumerate}
		With this in mind, it is now easy to see that we can perform exactly one t-switch in $T$ for each unordered pair of disjoint edges of $T$. Then, $\deg_f(T)=\dpe(d)$, and we apply \Cref{dpe.formula}.
	\end{proof}
	
	It is quite surprising that the f-degree of a tree is an invariant associated with its degree sequence. \Cref{deg_f.tree} has a significant impact on the global structure of $\mathcal{T}(d)$. Indeed, we see that every vertex of this graph has the same degree, namely $\dpe(d)$. Moreover, we already know that $\mathcal{T}(d)$ is connected (see \cite{schvollner_pseudoforests}). Hence, we obtain the following corollary.
	
	\begin{corollary}
		\label{T_d_regular}
		The graph $\mathcal{T}(d)$ is connected and $\dpe(d)$-regular.
	\end{corollary}
	
	\begin{proof}
		It follows from the preceding discussion.
	\end{proof}
	
	Despite what we have just observed for $\deg_f(T)$, the next corollary shows that $\deg(T)$ is not generally an invariant associated with $d(T)$. 
	
	\begin{corollary}
		\label{deg.tree.G(d)}
		If $T$ is a tree of order $n$, then
		\begin{equation}
			\deg(T) = 2\deg_f(T) - \sum_{uv \in E(T)} (d_u - 1)(d_v - 1) = (n - 1)^2 - \zeta_2(T).
		\end{equation}
	\end{corollary}
	
	\begin{proof}
		The first equality is an immediate consequence of \Cref{deg.general.formula,deg_f.tree}, where $k_3(T) = 0 = c_4(T)$ and $\dpe(T) = \deg_f(T)$. The second equality follows from \Cref{deg.girth>4}.
	\end{proof}
	
	If $T$ is a tree in $V(\mathcal{G}(d))$, note that $\deg(T) - \deg_f(T)$ is the number of disconnected neighbors of $T$ in $\mathcal{G}(d)$.\\
	
	We denote by $\mathcal{U}(d)$ the subgraph of $\mathcal{G}(d)$ induced by unicyclic graphs with degree sequence $d$. It is worth noting that $\mathcal{U}(d)$ is connected (see \cite{schvollner_pseudoforests}). A 2-switch $\tau$ on a unicyclic graph $U \in V(\mathcal{U}(d))$ is said to be a \emph{u-switch} on $U$ if $\tau(U) \in V(\mathcal{U}(d))$. The \emph{u-degree} of a unicyclic graph $U$, denoted by $\deg_u(U)$, is the degree of the vertex $U$ in the graph $\mathcal{U}(d)$. Equivalently, $\deg_u(U)$ is the number of u-switches acting on $U$. Clearly, $\deg_u(U) \leq \deg(U)$, because $\mathcal{U}(d) \preceq \mathcal{G}(d)$.
	
	In a unicyclic graph $U$, we can always identify the following two subgraphs: the unique cycle $C$ of $U$ and the forest  
	\[ F = U - E(C) - \{v \in V(C) : \deg_U(v) = 2 \}. \]
	Then, $E(U) = E(C) \dot{\cup} E(F)$ and $|V(F) \cap V(C)|$ is the number of components of $F$. If $d(U) = 2^n$, then $U \approx C_n$ and $F = K_0$. But if $d_v \neq 2$ for some $v \in V(U)$, then $E(F) \neq \varnothing$. Note that $C$ is an induced subgraph of $U$, but $F$, in general, is not.
	
	\begin{corollary}
		If $U \in \mathcal{U}(d)$ is a unicyclic graph of order $n$ with a cycle of length $c$, then 
		\begin{equation*}
			\deg(U) =
			\begin{cases}
				n^2 - \zeta_2(U) + 3, & \text{if } c = 3 \\
				n^2 - \zeta_2(U) + 2, & \text{if } c = 4 \\
				n^2 - \zeta_2(U), & \text{if } c \geq 5 \\
			\end{cases}
		\end{equation*}
	\end{corollary}
	
	\begin{proof}
		It follows immediately from \Cref{relacion.deg.zeta2}.
	\end{proof}
	
	\begin{theorem}[\cite{schvollner_pseudoforests}]
		\label{u-switch.caract}
		Let $\tau$ be a 2-switch acting on the edges $e_1$ and $e_2$ of a unicyclic graph $U$. If $C$ and $F$ are, respectively, the cycle and the forest of $U$, then we have the following:
		\begin{enumerate}
			\item if $e_1, e_2 \in F$ and $e \in E(C)$, then $\tau$ is a u-switch on $U$ if and only if it is a t-switch on $U - e$;
			\item if $e_1 \in C$ and $e_2 \in F$, then $\tau$ is a u-switch on $U$;
			\item if $e_1, e_2 \in C$, then $\tau$ is a u-switch on $U$ if and only if $\tau(C) \approx C$.
		\end{enumerate}
	\end{theorem}
	
\begin{theorem}\label{deg_u.unicyclic}
	Let $U$ be a unicyclic graph with cycle $C$ and forest $F$, and define $\zeta_C(U)=\sum_{uv \in E(C)} (d_u - 2)(d_v - 2)$ (where the degrees $d_u$, $d_v$ are taken in $U$). Then,
	\begin{align*}
		\deg_u(U) = & \deg(U) - \deg(C) + \dpe(C) - \dpe(F) + p_4(F) + \zeta_C(U) \\
		= & \deg(U) - \deg(C) - \deg(F) + \dpe(C) + \dpe(F) + \zeta_C(U).
	\end{align*}
\end{theorem}

\begin{proof}
	 As $E(U) = E(C) \dot{\cup} E(F)$, every 2-switch on $U$ falls into exactly one of three classes:
	\begin{enumerate}[(1).]
		\item those between edges of $C$;
		\item those between edges of $F$;
		\item those between an edge of $C$ and an edge of $F$.
	\end{enumerate}
	We count the u-switches in each class; $\deg_u(U)$ is the sum of the three counts. We will use repeatedly the following characterization (see Section 2 of \cite{schvollner_pseudoforests}): a 2-switch $\binom{a \ b}{c \ d}$ on a tree $T$ is a t-switch if and only if $T$ contains the path $ab \ldots cd$ or $ba \ldots dc$. 
	\begin{enumerate}[(1).]
		\item Notice that, for every $\{e,f\}\subseteq E(C)$ with $e\cap f=\varnothing$, we have exactly one 2-switch on $C$ that produces a new cycle isomorphic to $C$. The other possible 2-switch on $C$ replacing $\{e,f\}$, if applicable, splits $C$ into two disjoint cycles. Therefore, by \Cref{u-switch.caract}, there are $\dpe(C)$ u-switches of class (1).  
		
		 
		\item Let $T_1, \ldots, T_k$ be the components of $F$, and $r_i$ the unique vertex of $T_i$ in $V(C)$, called the \emph{root} of $T_i$. Fix $e \in E(C)$; then $U - e$ is a tree containing each $T_i$ as a subtree. If $e_1, e_2 \in E(T_i)$, the path of $U-e$ connecting them lies in $T_i$, so by the characterization above and \Cref{u-switch.caract} the u-switches acting on $T_i$ are exactly its t-switches: $\dpe(T_i)$ of them, by \Cref{deg_f.tree}. If instead $e_1 = ab \in E(T_i)$ and $e_2 = xy \in E(T_j)$ with $i \neq j$, choose $a$ (resp.\ $x$) in the root-component of $T_i - e_1$ (resp.\ $T_j - e_2$); then $b \neq r_i$, $y \neq r_j$, and $U-e$ contains the path $ba \cdots r_i \cdots r_j \cdots xy$. Of the two pairings, $\binom{a \ b}{y \ x}$ is always a t-switch on $U-e$ (since $b \neq r_i$ and $y \neq r_j$); the other, $\binom{a \ b}{x \ y}$, when applicable is not a t-switch on $U-e$ (its connecting path has the wrong form), and fails to be applicable if and only if $ax = r_ir_j \in E(C)$. By \Cref{u-switch.caract}, only the first is a u-switch; each edge $uv \in E(C)$ blocks $(d_u-2)(d_v-2)$ pairs. Summing over components via \Cref{degree.disconn.G,dpe.components}, class (2) contains $\deg(F) - \zeta_C(U)$ 2-switches, of which $\dpe(F)$ are u-switches.

		\item By \Cref{u-switch.caract}, every 2-switch of this class is a u-switch. Since classes (1) and (2) contain, respectively, $\deg(C)$ and $\deg(F) - \zeta_C(U)$ 2-switches, class (3) contains exactly $\deg(U) - \deg(C) - \big( \deg(F) - \zeta_C(U) \big)$ 2-switches, all of them u-switches.
	\end{enumerate}
	
	Adding the three counts, we obtain the second formula. The first one follows from \Cref{deg.girth>4} applied to $F$.
\end{proof}

	\begin{corollary}
		\label{deg_u.explicito}
		Let $U \in \mathcal{U}(d)$ be a unicyclic graph of order $n$ with cycle $C$ and forest $F$. If $c = |V(C)|$, then $\deg_u(U) =$
		\begin{equation*}
			\begin{cases}
				2\dpe(d) + \dpe(F) - \deg(F) - p_4(U)+ \zeta_C(U), & \text{if } c = 3 \\
				2\dpe(d) + \dpe(F) - \deg(F) - p_4(U) + \zeta_C(U)+ 2, & \text{if } c = 4 \\
				2\dpe(d) + \dpe(F) - \deg(F) - p_4(U) + \zeta_C(U)- \frac{c}{2}(c - 5), & \text{if } c \geq 5 \\
			\end{cases}
		\end{equation*}
		where
		\begin{equation*}
			p_4(U) =
			\begin{cases}
				\zeta_2(U) - |d|^2 + n - 3, & \text{if } c = 3 \\
				\zeta_2(U) - |d|^2 + n, & \text{if } c \geq 4 \\
			\end{cases}
		\end{equation*}
	\end{corollary}
	
	\begin{proof}
		The formulas for $\deg_u(U)$ are obtained by substituting the expressions for $\deg(U)$ from \Cref{deg.formula}, together with $\dpe(C) = \binom{c}{2} - c$, $\deg(C_3) = 0$, $\deg(C_4) = 2$, and $\deg(C_n) = n(n-4)$ for $n \geq 5$ (see \Cref{ciclo.deg}), into the second formula of \Cref{deg_u.unicyclic}. The formulas for $p_4(U)$ follow from \Cref{relacion.deg.zeta2} combined with $\deg(U) = 2\dpe(d) + 2c_4(U) - p_4(U)$ from \Cref{deg.formula}, noting that $|E(U)| = n$, $c_4(U) = 1$ if $c = 4$ and $0$ otherwise, and $k_3(U) = 1$ if $c = 3$ and $0$ otherwise.
	\end{proof}
	
	Contrary to $\mathcal{T}(d)$, $\mathcal{U}(d)$ is not regular in general. To see this, consider the unicyclic graph $U$ obtained by identifying a vertex of a triangle with a leaf of a $P_4$, and the unicyclic graph $U'$ obtained by identifying a vertex of a $C_4$ with a leaf of a $P_3$. We see that $d(U) = 3^1 2^4 1^1 = d(U')$. However, $\deg_u(U) = 11$ and $\deg_u(U') = 10$.

\section*{Acknowledgements}
This work was partially supported by Universidad Nacional de San Luis, grants PROICO 03-0723 and PROIPRO 03-2923, MATH AmSud, grant 22-MATH-02, Consejo Nacional de Investigaciones
Cient\'ificas y T\'ecnicas grant PIP 11220220100068CO and Agencia I+D+I grants PICT 2020-00549 and PICT 2020-04064. 

	
\end{document}